\documentclass[letterpaper, 10 pt, conference]{ieeeconf}  % Comment this line out if you need a4paper

\IEEEoverridecommandlockouts                              % This command is only needed if 
\usepackage{todonotes}
\usepackage{latexsym,amssymb,amsmath, amsbsy,amsopn, amstext}
\usepackage{graphicx,epsfig,tikz,pgf,hyperref,cancel,color,float,ifpdf}
\usepackage{extarrows,mathrsfs,cite}
\usepackage{amsmath,  amssymb,stfloats}
\usepackage{subcaption}
\usepackage{tikz}
\usetikzlibrary{automata, positioning, arrows}
\usepackage[ruled,vlined,linesnumbered]{algorithm2e}
\usepackage{bm}
\usepackage{etoolbox}
\makeatletter
\patchcmd{\@makecaption}
  {\scshape}
  {}
  {}
  {}
\makeatother
\usepackage{makecell}

\makeatletter
\patchcmd{\@algocf@start}% <cmd>
  {-1.5em}% <search>
  {0pt}% <replace>
  {}{}% <success><failure>
\makeatother

\graphicspath{{./figures/}}

\newtheorem{definition}{\bfseries Definition}%[section]
\newtheorem{example}{\bfseries Example}%[section]
\newtheorem{theorem}{\bfseries Theorem}
\newtheorem{corollary}{\bfseries Corollary}%[section]
\newtheorem{lemma}{\bfseries Lemma}%[section]

\newtheorem{remark}{\bfseries Remark}

\def\x{\bm{x}}

\def\u{\bm{u}}

\def\f{\bm{f}}
\def\g{\bm{g}}
\def\y{\bm{y}}
\newcommand{\mm}[1]{\mathcal{#1}}
\newcommand{\mt}[1]{\boldsymbol{#1}}

\newcommand{\li}{[\![}
\newcommand{\ri}{]\!]}

\newcommand{\bxi}{\bm{\xi}}
\newcommand{\z}{\bm{z}}
\newcommand{\cc}{\mathbf{c}}
\newcommand{\G}{\mathbf{G}}
\newcommand{\A}{\mathbf{A}}
\newcommand{\bb}{\mathbf{b}}

\newcommand{\Zc}{\mathcal{Z}}

\newcommand{\mat}[1]{\begin{bmatrix} #1 \end{bmatrix}}
\newcommand{\hz}[1]{\langle \G^c_{#1},\allowbreak \G^b_{#1},\allowbreak \cc_{#1},\allowbreak \A^c_{#1},\allowbreak \A^b_{#1},\allowbreak \bb_{#1} \rangle}
\newcommand{\overhz}[1]{\langle \overline{\G}^c_{#1},\allowbreak \overline{\G}^b_{#1},\allowbreak \overline{\cc}_{#1},\allowbreak \overline{\A}^c_{#1},\allowbreak \overline{\A}^b_{#1},\allowbreak \overline{\bb}_{#1} \rangle}

\allowdisplaybreaks[4]

\usepackage{xcolor}
\newif\ifdraft

\drafttrue

\title{\LARGE \bf
Hybrid Zonotope-Based Backward Reachability Analysis for Neural Feedback Systems With Nonlinear Plant Models}

\author{Hang Zhang, Yuhao Zhang and Xiangru Xu
\thanks{H. Zhang, Y. Zhang and X. Xu are with the Department of Mechanical Engineering, University of Wisconsin-Madison, Madison, WI, USA. Email:         {\tt\small \{hang.zhang,yuhao.zhang2,xiangru.xu\}@wisc.edu}.}%
}

\begin{document}
\maketitle
\begin{abstract}
The increasing prevalence of neural networks in safety-critical control systems underscores the imperative need for rigorous methods to ensure the reliability and safety of these systems. This work introduces a novel approach employing hybrid zonotopes to compute the over-approximation of backward reachable sets for neural feedback systems with nonlinear plant models and general activation functions.  
Closed-form expressions as hybrid zonotopes are provided for the over-approximated backward reachable sets, and a refinement procedure is proposed to alleviate the potential conservatism of the approximation.
Two numerical examples are provided to illustrate the effectiveness of the proposed approach.

\end{abstract}
% \hang{Things remaining:
% \begin{itemize}
%     \item Introduction and Preliminary
%     \item An extra example
%     \item some plots for the examples (see Section \ref{sec:ex}.)
% \end{itemize}}

\section{Introduction}\label{sec:intro}
The integration of Neural Networks (NNs) in feedback control systems necessitates the development of effective methods to analyze the system's behavior with robust and reliable performance guarantees. NNs demonstrate remarkable adaptability and are capable of learning intricate mappings from input to output, rendering them a potent tool for controlling nonlinear dynamical systems. Nevertheless, NNs are highly sensitive to even minor perturbations in the input space \cite{yuan2019adversarial}, potentially leading Neural Feedback Systems (NFSs), which are systems with NN controllers in the feedback loop, to safety hazards.
To analyze the behavior of such NFSs, various methods have been proposed for the reachability analysis of NFSs. The safety properties of these systems can be verified by the forward reachability analysis \cite{dutta2019reachability, huang2019reachnn, everett2021reachability, tran2020nnv, Fazlyab2022Safety, zhang22constrainedzonotope,zhang2022reachability,kochdumper2023open} or the backward reachability analysis \cite{vincent2021reachable,rober2023backward,zhang2023backward}. 
%have been proposed to ensure safety and comprehensively analyze the behavior of feedback systems with NN controllers, which is referred to as \emph{neural feedback systems} (NFS). 
%One approach is to find safety certificates for neural feedback systems, such as the Lyapunov function or the Barrier function \cite{chang2019neural,yin2021stability,dai2021lyapunov,dawson2023safe}. However, the behavior of feedback systems remains unclear when such safety certificates cannot be identified. 
%Another approach entails verifying the properties of neural feedback systems with reachability analysis techniques. 
%Some of the existing works focus on computing the forward reachable sets of NFS \cite{dutta2019reachability, huang2019reachnn, everett2021reachability, tran2020nnv, Fazlyab2022Safety, zhang22constrainedzonotope,zhang2022reachability}, while recently finding the backward reachable sets (BRSs) have attracted attention \cite{vincent2021reachable,rober2023backward,zhang2023backward}. 

The backward reachability problem is to compute a set of states, known as the Backward Reachable Set (BRS), from which the system’s trajectories can reach a specified target region within a finite time horizon. When the target region is chosen as the unsafe set for safety verification problems, the computed BRS will bound the back-propagated trajectories that enter the unsafe set.  Although there exist various methods for computing the BRSs of systems without NNs \cite{mitchell2005time,yang2022efficient}, they are not directly applicable to NFSs due to the inherently complex and nonlinear nature of NNs. 
For isolated ReLU-activated NNs, \cite{vincent2021reachable} proposed a method to compute the exact BRS of NNs by representing the NNs as piecewise linear functions via the activation patterns of the ReLU functions. For the NFS with a linear plant model, a method based on Linear Programming (LP) relaxation was proposed in \cite{rober2022backward} to over-approximate the BRSs; this method was extended in \cite{rober2023backward} for NFSs with nonlinear plant models by utilizing the over-approximation tool OVERT \cite{sidrane2022overt} and a guided partition algorithm. Despite the effectiveness of these LP-based methods, the computed BRSs are usually limited in the form of intervals and tend to be conservative.
% This target region is usually chosen as the unsafe set for safety verification problems. Then, the computed BRSs bound the back-propagated trajectories that enter the unsafe set.  For general controlled feedback systems, the computation of the BRSs has been investigated in \cite{mitchell2005time,yang2022efficient}. However, they are not directly applicable to NFS due to inherent complexity and highly nonlinear nature of NNs. 
% For isolated ReLU-activated NNs, \cite{vincent2021reachable} proposed to represent the NNs as piecewise linear functions by identifying the activation patterns of the activation functions. For the NFS with linear system dynamics, a method based on Linear Programming (LP) relaxation was proposed in \cite{rober2022backward} to over-approximate the BRSs. And \cite{rober2023backward} extended the LP-based method and proposed an algorithm to compute BRSs for
% NFS with nonlinear dynamics by utilizing the over-approximation tool OVERT \cite{sidrane2022overt}. Despite the effectiveness of these LP-based methods, the computed BRSs are usually limited to be in the form of intervals.
% %An exact BRS computation method was presented in \cite{zhang2023backward} using set propagation techniques of hybrid zonotopes (HZs). 

Recently, Hybrid Zonotope (HZ) was proposed as a new set representation that generalizes the polytope or constrained zonotope \cite{bird2023hybrid}. By including binary variables in the set representation, an HZ is equivalent to a finite union of polytopes, and therefore, can conveniently represent non-convex and disconnected sets with flat faces. HZs have proven advantages for set-based computations in terms of computational efficiency and accuracy 
%as many standard set operations can be computed exactly through simple identities 
\cite{bird2023hybrid,bird2021unions,bird2022hybrid,ortiz2023hybrid,siefert2023successor,siefert2023reachability}. In particular, HZs possess properties that make them well-suited for reachability analysis of NFSs by using set operations. For example, \cite{zhang2022reachability} showed that a Feed-forward Neural Network (FNN) with ReLU activation functions can be exactly represented by an HZ and provided algorithms to compute the exact and approximated
forward reachable sets (FRSs) of NFSs; \cite{zhang2023backward} considered NFSs with linear plant models and presented results for computing their exact BRSs in the form of HZs. 
%ReLU-activated feed-forward neural network can be exactly represented by a HZ. 

%This work builds on our previous work \cite{zhang2023backward}, which considered NFSs with linear dynamics and presented results for computing their exact BRSs in the form of HZs. 
%However, prior work cannot be directly applied to the backward reachability of NFS with nonlinear dynamics. 
This work aims to compute the HZ representations of the over-approximated BRSs for NFSs where the plant model is nonlinear and the controller is an FNN with general activation functions. The contributions of this work are at least twofold: (i) By leveraging over-approximation tools of nonlinear functions, closed-form identities are provided for over-approximating the BRSs of NFSs with general nonlinear plant models and ReLU-activated FNN controllers; 
%(ii) A refinement technique is proposed to reduce the conservatism induced by the over-approximation of nonlinear dynamics; 
(ii) The proposed method is extended to NNs with other commonly used activation functions such as leaky ReLU, tanh, and sigmoid. %Two numerical examples are provided to demonstrate the performance of the proposed method.
%\rev{To the best of our knowledge, the proposed method is the first set-propagation-based method that over-approximates the BRSs of an NFS with nonlinear plant models in closed forms.}
%\rev{To the best of our knowledge, the proposed method is the first set-propagation-based approach that addresses the backward reachability analysis of NFSs with nonlinear plant models.} 
The remainder of this paper is organized as follows: Section \ref{sec:pre} provides preliminaries on HZ and the problem statement; Section \ref{sec:BRS} presents the results for NFSs with nonlinear plant models and ReLU-activated FNNs; Section \ref{sec:other_activate} extends these results to NFSs with general activation functions; two numerical examples are provided in Section \ref{sec:ex} to demonstrate the performance of the proposed method; and finally, the concluding remarks are made in Section \ref{sec:concl}.

\emph{Notation.} Vectors and matrices are denoted as bold letters (e.g., $\x \in \mathbb{R}^n$ and $\A \in \mathbb{R}^{n\times n}$). The $i$-th component of a vector $\x\in \mathbb{R}^n$ is denoted by $x_i$ with $i\in\{1,\dots,n\}$. For a matrix $\A\in \mathbb{R}^{n\times m}$, $\A[i:j,:]$ denotes the matrix constructed by the $i$-th to $j$-th rows of $\A$. The identity matrix in $\mathbb{R}^{n\times n}$ is denoted as $\bm{I}_n$ and $\bm{e}_i$ is the $i$-th column of $\bm{I}$. The vectors and matrices whose entries are all 0 (resp. 1) are denoted as $\bm{0}$ (resp. $\bm{1}$). Given sets $\mathcal{X}\subset \mathbb{R}^n$, $\mathcal{Z}\subset \mathbb{R}^m$ and a matrix $\bm{R}\in\mathbb{R}^{m\times n}$, the Cartesian product of $\mathcal{X}$ and $\mathcal{Z}$ is $\mathcal{X}\times \mathcal{Z} = \{(\x,\z)\;|\;\x\in\mathcal{X},\z\in\mathcal{Z}\}$, the Minkowski sum of $\mm{X}$ and $\mm{Z}$ is $\mm{X} \oplus \mm{Z} = \{ \mt{x}+\mt{z} \mid \mt{x} \in \mm{X}, \mt{z} \in \mm{Z} \}$, the generalized intersection of $\mathcal{X}$ and $\mathcal{Z}$ under $\bm R$ is $\mathcal{X} \cap_{\bm R}\mathcal{Z} = \{\x\in\mathcal{X}\;|\;\bm R \x\in\mathcal{Z}\}$, and the $k$-ary Cartesian power of $\mathcal{X}$ is $\mathcal{X}^k = {\mathcal{X}\times\cdots\times\mathcal{X}}$. Given $\underline{\mt{a}}$, $\overline{\mt{a}} \in \mathbb{R}^n$ such that $\underline{\mt{a}}\leq \overline{\mt{a}}$ where $\leq $ is in the element-wise sense, the set $\{\mt{x} \in \mathbb{R}^n\mid \underline{\mt{a}} \leq \mt{x} \leq \overline{\mt{a}}\}$ is denoted as $\li \underline{\mt{a}}, \overline{\mt{a}}\ri$. 
Given two scalar $\underline{a}$, $\overline{a} \in \mathbb{R}$ such that $\underline{a} \leq \overline{a}$, the interval $\{{x} \in \mathbb{R} \mid \underline{{a}} \leq {x} \leq \overline{{a}}\}$ is denoted as $\li \underline{{a}}, \overline{a} \ri$.

%\XX{Yuhao, please rewrite the intro. I don't think we need to talk too much about safety-critical or autonomous system stuff as we don't plan to include a formal section about safety verification; we don't have to talk about safety certificate for neural feedback systems either; focus on the forward and backward reachability analysis of NFS. Talk about HZ in the intro as well - I did not find any intro to HZs in the previous version. What are the pros of HZs? Why we choose to use HZs? Explain previous results on BRS computation, especially for nonlinear systems. The intro can be 3/4 pages considering the page limitation.}

%\XX{define acronym neural feedback system (NFS)?}

%The remainder of this paper is organized as follows. Section \ref{sec:pre} provides preliminaries on hybrid zonotopes and problem formulations. Section \ref{sec:BRS} presents the over-approximation techniques used for the nonlinear dynamics and the neural feedback system and provides an algorithm to over-approximate the BRSs when the NNs are activated by ReLU functions. Section \ref{sec:other_activate} shows how to extend the proposed method for NNs with other types of activation functions. Two numerical examples are provided in Section \ref{sec:ex} to demonstrate the performance of the proposed method and concluding remarks are made in Section \ref{sec:concl}.

\section{Preliminaries \& Problem Statement}\label{sec:pre}

%\textbf{Notation:} 
 
%\XX{where is this notation from?}\hang{the usage of \li \ri is to distinguish the indices of matrix and vector. See \cite{yang2022efficient} for the notation}.  
%In this letter, scalars are denoted by lowercase letters (e.g., $\alpha \in \mathbb{R}$), vectors are denoted by bold lowercase letters (e.g., $\x \in \mathbb{R}^n$), matrices are denoted by bold capital letters (e.g., $\A \in \mathbb{R}^{n\times n}$) and sets are denoted by calligraphic letters (e.g., $\X \subset \mathbb{R}^{n\times n}$).

\subsection{Hybrid Zonotopes}\label{sec:zono}

\begin{definition}\cite[Definition 3]{bird2023hybrid}\label{def:sets}
%Let $\Zc,\Zc_c,\Zc_h \subset\mathbb{R}^n$. $\Zc$ is a \emph{zonotope} if \eqref{equ:zono} holds \cite{mcmullen1971zonotopes}, $\Zc_c$ is a \emph{constrained zonotope} if \eqref{equ:czono} holds \cite{scott2016constrained}, and $\Zc_h$ is a \emph{hybrid zonotope} if \eqref{equ:hzono} holds \cite{bird2023hybrid}:
The set $\Zc \subset\mathbb{R}^n$ is a \emph{hybrid zonotope} if there exist $\mathbf{c} \in \mathbb{R}^{n}$, $\mathbf{G}^c \in \mathbb{R}^{n \times n_{g}}$, $\G^b\in \mathbb{R}^{n \times n_{b}}$, $\A^c \in \mathbb{R}^{n_{c}\times{n_g}}$, $\A^b \in\mathbb{R}^{n_{c}\times{n_b}}$, $\bb \in \mathbb{R}^{n_{c}}$ such that
\begin{align*}
%    &\exists (\mathbf{c}, \mathbf{G}) \in  \mathbb{R}^{n}\times \mathbb{R}^{n \times n_{g}}: \!\Zc =\left\{\mathbf{G} \bm{\xi}+\mathbf{c} \;|\; \|\bm{\xi}\|_{\infty} \leq 1\right\}, \label{equ:zono}\\
%    &\exists (\mathbf{c}, \mathbf{G},\A,\bb) \in  \mathbb{R}^{n}\times \mathbb{R}^{n \times n_{g}}\times \mathbb{R}^{n_{c} \times n_{g}} \times \mathbb{R}^{n_{c}}: \notag\\
%    &  \; \quad\quad\quad\quad\quad\quad\Zc_c=\left\{\mathbf{G} \bm{\xi}+\mathbf{c} \;|\; \|\bm{\xi}\|_{\infty} \leq 1, \mathbf{A} \bm{\xi}=\mathbf{b}\right\},\label{equ:czono}\\
%    &\exists (\mathbf{c}, \mathbf{G}^c,\G^b,\A^c,\A^b,\bb) \in  \mathbb{R}^{n}\!\times\! \mathbb{R}^{n \times n_{g}}\!\times\! \mathbb{R}^{n \times n_{b}} \!\times\! \mathbb{R}^{n_{c}\times{n_g}}\notag \\
%    & \!\quad\quad\quad\quad\quad\quad\quad\quad\quad \quad\quad\times \mathbb{R}^{n_{c}\times{n_b}}\times \mathbb{R}^{n_{c}}: \label{equ:hzono}\\ \notag
    & \mathcal{Z}\!=\!\left\{\mat{\G^c \!\! & \!\!\G^b}\mat{\bm{\xi}^c\\ \bm{\xi}^b}+\cc \left |\!\! \begin{array}{c}
{\mat{\bm{\xi}^c\\ \bm{\xi}^b} \in \mathcal{B}_{\infty}^{n_{g}} \times\{-1,1\}^{n_{b}}}, \\
{\mat{\A^c \!\!& \!\!\A^b}\mat{\bm{\xi}^c\\ \bm{\xi}^b}=\bb}
\end{array}\right.\!\!\!\right\}
\end{align*}
where $\mathcal{B}_{\infty}^{n_g}=\left\{\bm{x} \in \mathbb{R}^{n_g} \;|\;\|\bm x\|_{\infty} \leq 1\right\}$ is the unit hypercube in $\mathbb{R}^{n_{g}}$. The \emph{HCG}-representation of the hybrid zonotope 
%in Hybrid Constrained Generator-representation (HCG-rep) 
is given by $\mathcal{Z}= \hz{}$.
%The shorthand notations of the zonotope, constrained zonotope and hybrid zonotope  are given by $\Zc = Z\langle \cc, \G\rangle$, $\mathcal{Z}_c= CZ\langle \mathbf{c}, \mathbf{G}, \mathbf{A}, \mathbf{b}\rangle$, and $\mathcal{Z}_{h}= HZ\langle \mathbf{c}, \mathbf{G}^c, \mathbf{G}^b, \mathbf{A}^c, \mathbf{A}^b, \mathbf{b}\rangle$, respectively. 
\end{definition}

%Note that a hybrid zonotope degenerates into a constrained zonotope when $n_b=0$, and a constrained zonotope degenerates into a zonotope when $n_c=0$sss. 
Given an HZ $\mathcal{Z} = \hz{}$, the vector $\cc$ is called the center, the columns of $\G^b$ are called the \emph{binary generators}, and the columns of $\G^c$ are called the \emph{continuous generators}. The representation complexity of the HZ $\mm{Z}$ is determined by the number of continuous generators $n_g$, the number of binary generators $n_b$, and the number of equality constraints $n_c$ \cite{bird2023hybrid}. 
For simplicity, we define the set $\mathcal{B}(\A^c,\A^b,\bb) \triangleq \{(\bm{\xi}^c,\bm{\xi}^b) \in \mathcal{B}^{n_g}_\infty \times \{-1,1\}^{n_b} \;|\; \A^c\bm{\xi}^c + \A^b\bm{\xi}^b = \bb \}$. 
An HZ with $n_b$ binary generators is equivalent to the union of $2^{n_b}$ constrained zonotopes \cite[Theorem 5]{bird2023hybrid}. 
Identities are provided to compute the linear map and generalized intersection \cite[Proposition 7]{bird2023hybrid}, the union operation \cite[Proposition 1]{bird2021unions}, and the Cartesian product of HZs \cite[Proposition 3.2.5]{bird2022hybrid}.

\subsection{Problem Statement}\label{sec:prob}

Consider the following discrete-time nonlinear system:
\begin{equation}\label{dt-sys}
    \x{(t+1)} = \bm f (\x(t),\u(t)) % + w(t)
\end{equation}
where $\x(t)\in  \mathbb{R}^n,\; \u(t)\in \mathbb{R}^m$ are the state and the control input, respectively. We assume $\x\in\mathcal{X}$ and $\u\in\mathcal{U}\triangleq \li \underline{\u},\overline{\u} \ri$, where $\mathcal{X}\subset \mathbb{R}^n$ is called the state set and $\mathcal{U}\subset \mathbb{R}^m$ is called the input set. %$\bm A_d \in \mathbb{R}^{n \times n}$, $\bm B_d \in \mathbb{R}^{n \times m}$ are the state matrix and the input matrix, respectively. 
%With an initial state $\x(0)$ and a control sequence $\u = \u(0), \u(1), \dots$, the state trajectory of system \eqref{dt-sys} is denoted as $\x = \x(0),\x(1),\dots$.
The control input $\u$ is given as 
%$\u(t) = \pi(\x(t))$ 
\begin{equation}
\u(t) = \mt{\pi}(\x(t))\label{dt-input}    
\end{equation}
where $\pi$ is an $\ell$-layer FNN. The $k$-th layer weight matrix and the bias vector of $\pi$ are denoted as $\bm W^{(k-1)}$ and $\bm v^{(k-1)}$, respectively, where $k=1,\dots,\ell$. Denote $\x^{(k)}$ as the neurons of the $k$-th layer and $n_{k}$ as the dimension of $\x^{(k)}$. Then, for $k=1,\dots,\ell-1$, %we have
%The $k$-th ($k=1,\dots,\ell-1$) layer of the neuron of the FNN is given by
%for each layer $k=1,\dots,\ell-1$, the neuron of the FNN is given by
\begin{align}\label{equ:NN}
\x^{(k)}=\mt{\phi}(\bm W^{(k-1)}\x^{(k-1)}+ \bm v^{(k-1)})    
\end{align}
%where $\x^{(0)} = \x(t)$ and $\phi: \mathbb{R}^{n_{k}}\rightarrow \mathbb{R}^{n_{k}}$ is the vector-valued activation function constructed by component-wise repetition of ReLU function, i.e., 
where $\x^{(0)} = \x(t)$ and $\mt{\phi}$ is the vector-valued activation function constructed by component-wise repetition of the activation function $\sigma(\cdot)$, i.e., 
$%\phi(\x) \triangleq [ReLU(x_1) \;\cdots\; ReLU(x_{n_{k}})]^\top,\;\x\in\mathbb{R}^{n_{k}}.
\mt{\phi}(\x) \triangleq [\sigma(x_1) \;\cdots\; \sigma(x_{n})]^\top.$ 
In the last layer, only the linear map is applied, i.e., $$\mt{\pi}(\x(t)) = \x^{(\ell)} = \bm W^{(\ell-1)}\x^{(\ell-1)}+\bm v^{(\ell-1)}.$$ The activation function considered in this work is not restricted to ReLU; see Section \ref{sec:other_activate} for more discussion. 
In the following, we assume the FNN controller satisfies 
%the input constraints over the input set, i.e., 
$\mt{\pi}(\x) \in \mm{U}$, $\forall \x\in\mm{X}$; this assumption is not restrictive since the output of any FNN can be saturated into a given range by adding two additional layers with ReLU activation functions \cite{everett2021reachability,zhang2023backward}.%\XX{check}.

The NFS consisting of system \eqref{dt-sys} and the controller \eqref{dt-input}   is a closed-loop system denoted as:
\begin{equation}\label{close-sys}
    \x{(t+1)} = \f_{cl}(\x(t)) \triangleq \bm f(\x(t),\mt{\pi}(\x(t))).
\end{equation}

%\begin{definition} %($t$-step backward reachable set)
Given a target set $\mathcal{T} \subset \mathcal{X}$ for the system \eqref{close-sys}, the set of states that can be mapped into the target set $\mathcal{T}$ by \eqref{close-sys} in exactly $T$ steps is defined as the $T$-step BRS and denoted as $\mathcal{P}_T(\mathcal{T})\triangleq\{\x(0)\in \mathcal{X} \mid x(t) = \f_{cl}(\x(t-1)),  x(T)\in \mathcal{T}, t=1,2,\dots, T\}.$ 
% \begin{align} \label{eq:def_brs}
%     \mathcal{P}_k(\mathcal{T})\triangleq\{\x(0)\in \mathcal{X} \mid & x(k) = \f_{cl}(\x(k-1)),  \nonumber \\
%     & x(k)\in \mathcal{T}, k=1,2,\dots,K\}.
% \end{align}
%\end{definition}
%we denote $\mathcal{P}_t(\mathcal{T})\triangleq\{\x(0)\in \mathcal{X} | x(k)\in \mathcal{T}, x(k) = \f_{cl}(\x(k-1)) ,k=1,2,\dots,t\}$ the , i.e., .
For simplicity, the one-step BRS is also denoted as $\mathcal{P}(\mathcal{T})$, i.e., $\mathcal{P}(\mathcal{T}) = \mathcal{P}_1(\mathcal{T})$. In this work, we
assume the state set $\mathcal{X}$, the input set $\mathcal{U}$, the target set $\mathcal{T}$, and the set that over-approximates the nonlinear plant model $\bm f$ are all represented as HZs; this assumption enables us to handle sets and plant models using a unified HZ-based approach. 

%In order to apply an HZ-based approach for the BRS computation, we assume the state set $\mathcal{X}$, the input set $\mathcal{U}$ and the target set $\mathcal{T}$ are represented as HZs in this work.
%for the closed-loop system \eqref{close-sys} are represented by HZs. 

The following problem will be investigated in this work: 
%\begin{problem}\label{prob1}
\emph{Given a target set $\mathcal{T}\subset \mathcal{X}$ represented as an HZ and a time horizon $T \in \mathbb{Z}_{>0}$, compute the over-approximation of BRS $\mathcal{P}_t(\mathcal{T})$ of the neural feedback system \eqref{close-sys}, for $t=1,2,\dots,T$.}
%\end{problem}
% \begin{remark}
%     Although the system \eqref{close-sys} is actually a special case of the general nonlinear control systems $\mt{f}(\mt{x},\mt{u})$, the linear addition between the nonlinear dynamics $\mt{f}(\mt{x})$ and control input $\mt{u}$ in \ref{close-sys} is representative in practice due to the powerful abilities of neural networks to approximate a highly nonlinear function implicitly.\hang{Is this remark okay?}
% \end{remark}

Compared with our previous result \cite{zhang2023backward} that computes the exact BRS for NFSs consisting of linear plant models and ReLU-activated FNN controllers, this work considers NFSs consisting of general nonlinear plant models and FNNs with general activation functions, and therefore, only over-approximated BRSs are computed.%\XX{check} \hang{I used ``HZ'' instead of ``HZ-based''.}
% To solve the problem, we will compute an over-approximation of $\bm f$ and the representation of the graph of FNNs using HZs (Section \ref{subsec:nonlin}), and based on that, compute the over-approximation of the one-step BRS 
% %$\overline{\mm{P}}(\mm{T}) \supseteq \mm{P}(\mm{T})$ 
% in closed form (Section \ref{subsec:over_brs}).

\section{Backward Reachability Analysis of ReLU Activated Neural Feedback Systems}\label{sec:BRS}
Our proposed method includes three main steps: (i) compute an HZ envelope %\XX{use the term envelope in Sec. \ref{subsec:nonlin}?} 
of the nonlinear function $\bm f$; (ii) find an HZ graph representation of the input-output relationship of the FNN control policy $\mt{\pi}$; (iii) calculate the one-step over-approximated BRS in closed form, and refine the set when an accurate result takes precedence over computational efficiency. In this section, we assume the controller $\u$ shown in \eqref{dt-input}   
is an FNN $\mt{\pi}$ with ReLU activation functions across all the layers, i.e., $\sigma(x) = ReLU(x) = \max(x,0)$; this assumption will be relaxed in the next section. Note that the time index of $\mt{x}$ and $\mt{u}$ are omitted for clarity in the section.  %\XX{We should use $\sigma$ not $\phi$, right?} \hang{Yes I will modify the notion in Section \ref{sec:other_activate}}

\subsection{Envelope of Nonlinear Functions} \label{subsec:nonlin}
For NFSs with linear plant models and ReLU activation functions, HZ-based methods have been proposed to compute both the exact FRSs and the exact BRSs \cite{zhang2022reachability,zhang2023backward,ortiz2023hybrid}. However, since HZs can only represent sets with flat faces (i.e., unions of polytopes), these exact analysis methods are inapplicable to the case of general nonlinear plant models shown in \eqref{close-sys}. 
%the case when general nonlinear functions appear in the system dynamics as in \eqref{close-sys}. 
As the first step towards computing the over-approximated BRS of the system \eqref{close-sys}, we will find an envelope of the nonlinear function $\bm f$. 

We denote the \emph{graph set} of a given function $\f$ over the domain $\mm{X}\times \mm{U}$ as $\mm{G}_f(\mm{X},\mm{U}) \triangleq \{[\x^\top \ \u^\top \ \y^\top]^\top\;|\;\y = \f(\x,\u),\x\in\mm{X},\u\in\mm{U}\}$. Since the exact representation of the graph set $\mm{G}_f(\mm{X},\mm{U})$ is difficult to find, we will compute an HZ-represented over-approximation of $\mm{G}_f(\mm{X},\mm{U})$, denoted as $\overline{\mm{H}}_f(\mm{X},\mm{U}) $, such that $\overline{\mm{H}}_f(\mm{X},\mm{U}) \supseteq \mm{G}_f(\mm{X},\mm{U})$. The envelope of a nonlinear function over a given domain can be found using off-the-shelf methods  \cite{mccormick1976computability,caratzoulas2005trigonometric,wanufelle2007global}. In this work, we make use of two approximation tools, namely Special Ordered Set (SOS) \cite{beale1970special} and OVERT \cite{sidrane2022overt}, mainly because the envelopes obtained by these two methods are easy to compute and can be readily incorporated with the HZ-based set computations.
%their over-approximations are competitive with other methods and can be naturally incooperated with the HZ-based set computations.%transformed into HZ representations through simple identities \XX{better reason?}. %A comparison between the SOS-based approximation method and other over-approximation techniques can be found in \cite{wanufelle2007global}.

%\emph{1) SOS.} 
The over-approximation method based on SOS was originally developed for solving nonlinear and nonconvex optimization problems \cite{beale1970special}. In recent work, SOS approximations were utilized to compute the FRSs of nonlinear dynamical systems \cite{siefert2023successor}. %,siefert2023reachability}. %For a scalar-valued nonlinear function, its SOS approximation is defined as follows.
%\begin{definition}\label{def:SOS} \cite[Def. 3]{siefert2023successor} 
The SOS approximation $\mathcal{S}$ of a scalar-valued function was defined in \cite[Definition 3]{siefert2023successor}, 
%Specifically, an SOS approximation $\mathcal{S}$ of a scalar-valued function $f(x)$ is defined by a vertex matrix $V=\left[\bm \nu_1, \bm \nu_2, \ldots, \bm \nu_{n_\nu}\right] \in \mathbb{R}^{(n+1) \times n_\nu}$ such that $\bm \nu_i=\left(x_i, f\left(x_i\right)\right)$ and is given by $\mathcal{S}=\left\{V \mt{\lambda} \mid \mathbf{1}^T \bm \lambda=1, \bm 0 \leq \bm \lambda\right.$, where at most $n+1$ entries of $\bm \lambda \in \mathbb{R}^{n_\nu}$ are nonzero and correspond to an $n$-dimensional simplex$\}$ \cite[Definition 3]{siefert2023successor}. 
%\end{definition}
%For a vector-valued nonlinear function $\bm f = [f_1,f_2,\dots,f_n]^\top$, its SOS approximation can be found for each component $f_i$ separately. 
while the identity that converts $\mathcal{S}$ into an HZ was provided in \cite[Theoem 4]{siefert2023successor}. The maximum SOS approximation error $\mt{\delta}$ for some common nonlinear functions was provided in \cite{wanufelle2007global}. Given an SOS approximation $\mm{H}_{SOS}$ as an HZ and the maximum approximation error $\mt{\delta} \in \mathbb{R}^n$, an envelop of the function $\bm f$ can be given as an HZ $\overline{\mm{H}}_f(\mm{X},\mm{U}) \triangleq \mm{H}_{SOS} \oplus \mm{E}$, where $\mm{E} = \langle \text{diag}([\mt{0}^\top_{n+m} \ \mt{\delta}^\top]),\emptyset, \mt{0}_{2n+m},\emptyset,\emptyset,\emptyset \rangle$ is the error HZ. For nonlinear functions with multidimensional inputs, decomposition techniques can be applied to avoid the exponential complexity of SOS approximations \cite{wanufelle2007global}.

%The first approach relies on Special Ordered Set (SOS) approximations which were developed to replace nonlinear functions with piecewise-linear approximations \cite{beale1970special}. An identity to convert SOS approximations to hybrid zonotopes
%was firstly provided in \cite{siefert2023successor} and later extended to the union of V-rep polytopes in \cite{siefert2023reachability}. The maximum approximation error $\mt{\delta}$ for some nonlinear functions are provided in \cite{wanufelle2007global}. For differentiable nonlinear functions, the approximation error can be computed by evaluating the extreme values and the values at the boundary points of error functions over each piece-wise linear domain. Given the SOS approximation $\mm{H}_{SOS}$ and the maximum approximation error $\mt{\delta}$, the nonlinear functions can be over-approximated by $\overline{\mm{H}}_f(\mm{X},\mm{U}) = \mm{H}_{SOS} \oplus \mm{E}$, where $\mm{E} = \langle \text{diag}([\mt{0}^\top \ \mt{\delta}^\top]), \mt{0}  \rangle$ is the error zonotope.

%\emph{2) OVERT.} 
% The second approach is based on the approximation tool OVERT \cite{sidrane2022overt}. For the nonlinear function $\f$ with both multidimensional inputs and multidimensional outputs, the functional rewriting technique presented in Algorithm 1 of \cite{sidrane2022overt} can be utilized to decompose $\f$ into multiple elementary functions $e(x)$ with single inputs and single outputs.
%, such as $\sin(x)$, $x^2$ and $ \log(x)$. 
For the OVERT method, a nonlinear function $\f$ with multidimensional inputs and outputs will be decomposed into multiple elementary functions $e(\cdot)$ with single inputs and single outputs by the functional rewriting technique, and then over-approximations of the element functions will be constructed using a set of optimally chosen piecewise linear upper and lower bounds \cite{sidrane2022overt}. %For ease of readability\XX{better statement?}, 
In this work, we assume that the number of breakpoints of the piecewise linear upper and lower bounds, denoted as $N$ ($N\geq 3$), are the same and only optimize the breakpoints along one of the piecewise linear bounds using OVERT. 
Denote the set of breakpoints for the piecewise linear upper (resp., lower) bound as $\{(x_i,y_i^{ub})\}_{i=1}^N$ (resp., $\{(x_i,y_i^{lb})\}_{i=1}^N$). Note that the $ x$-coordinates of the lower and upper bound are the same. Since the endpoints satisfy $y_1^{lb} = y_1^{ub} = e(x_1)$ and $y_N^{lb} = y_N^{ub} = e(x_N)$, 
the set bounded by the piecewise linear upper and lower bounds can be seen as a union of $N-1$ polytopes.  Each polytope can be constructed by the breakpoints of the upper and lower bounds, leading to $N-1$ V-rep polytopes that over-approximates the nonlinear element functions. Using \cite[Theorem 5]{siefert2023reachability}, an HZ envelope representing the union of V-rep polytopes can be found.% through simple identities.

\begin{figure}[t] 
\centering
     \begin{subfigure}{0.237\textwidth}
         \centering
         \includegraphics[width=\textwidth]{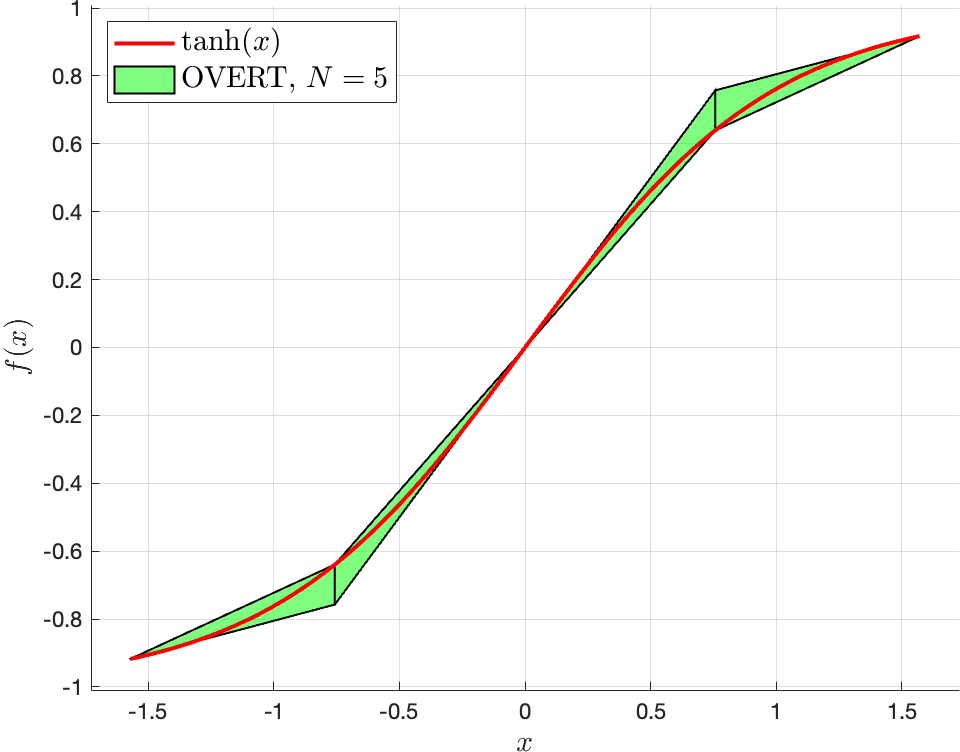}
         % \caption{The over-approximation of $\tanh(x)$ by OVERT.}
         % \label{fig:y equals x}
     \end{subfigure}
     \begin{subfigure}{0.237\textwidth}
         \centering
         \includegraphics[width=\textwidth]{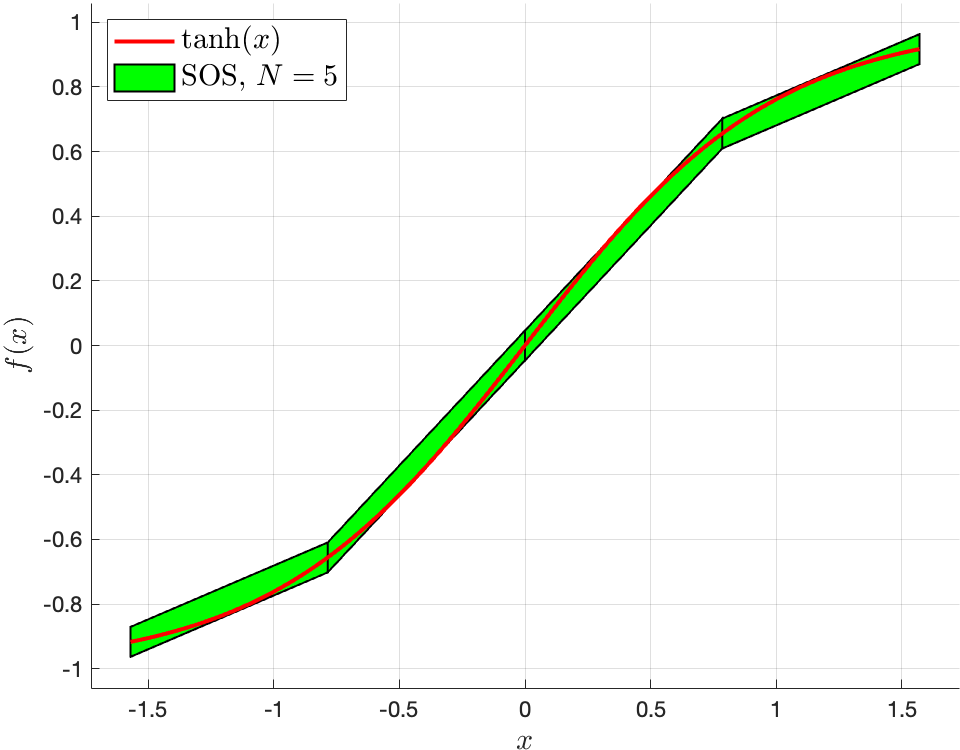}
         % \caption{The over-approximation of $\tanh(x)$ by SOS.}
         % \label{fig:three sin x}
     \end{subfigure}
    
    \caption{Comparison between OVERT and SOS. Both over-approximate $\tanh(x)$ for $x \in \li -\pi/2, \pi/2\ri$ with the union of 4 polytopes.}%\XX{Crop the figures to remove the blank margins.}}
    \label{fig:SOS_OVERT}
    \vskip -5mm
     % The over-approximation by OVERT (Left) is less conservative than that by SOS (right).
\end{figure}

%\emph{3) SOS vs. OVERT.} 
While both SOS-based and OVERT-based methods can utilize decomposition techniques to generate a tight over-approximation of $\bm f$, the resultant HZ envelopes exhibit varying levels of approximation accuracy and set complexity. For example, to compare the set complexity of the resulting HZs, we let the number of breakpoints be $N$ for both the SOS and the OVERT approaches. Then, for any one-dimensional nonlinear elementary function $e(x)$, the complexity 
%\XX{Did we introduce the complexty of HZ above?} 
of the HZ representations computed by SOS and OVERT is given by
\begin{align*}
n^{SOS}_g &   = 2N+2,\; n^{SOS}_b = N-1,\; n^{SOS}_c = N+4 ,\\
     n^{OVERT}_g  & = 4N-4,\; n^{OVERT}_b = N-1,\; n^{OVERT}_c = 2N.
\end{align*}
It can be observed that the HZ representation based on OVERT approximations tends to have a larger set complexity than the SOS method in terms of more continuous generators and more equality constraints. However, the OVERT method usually leads to more accurate over-approximations than the SOS method. This is because the approximation error $\mt{\delta}$ in the SOS method remains constant across the entire domain, whereas the approximation error in the OVERT method varies based on the distance from the breakpoints. An example of over-approximation of the $\tanh$ function using both methods is shown in Figure \ref{fig:SOS_OVERT}. Comparisons of the two methods will be further discussed in Section \ref{sec:ex}. 

\subsection{Over-approximation of One-step BRS}\label{subsec:over_brs}

In this subsection, we will compute the closed-form expression of an over-approximation of the one-step BRS, utilizing the HZ envelope of $\bm f$ calculated in Section \ref{subsec:nonlin}. %Given that the HZ envelope provides an over-approximation of the input-output relationship of the system dynamics $\mathbf{f}$, our analysis will result in an over-approximated representation of the exact one-step BRS.
%Since the HZ-based envelope can only over-approximate the input-output relationship of $\f$, we will also only be able over-approximate the exact one-step BRS $\mm{P}$. 

In order to incorporate the FNN controller $\mt{\pi}$ into the HZ-based framework, we apply Algorithm 1 in our previous work \cite{zhang2023backward} to compute an HZ-represented graph set of $\mt{\pi}$ over the domain $\mathcal{X}$, denoted as $\mathcal{H}_\pi = \hz{\pi}$, i.e., 
\begin{align*}
\mathcal{H}_\pi = \mathcal{G}_\pi(\mathcal{X})\triangleq \{[\x^\top \ \u^\top]^\top \;|\; \u = \mt{\pi}(\x),\x\in\mm{X}\}.
\end{align*}
%$\mathcal{H}_\pi = \mathcal{G}_\pi(\mathcal{X})\triangleq \{[\x^\top \ \u^\top]^\top \;|\; \u = \mt{\pi}(\x),\x\in\mm{X}\}$. 
%In Section \ref{subsec:nonlin}, the HZ-based envelope of nonlinear dynamics $\f$ has been obtained. In addition to that, the other essential step of our proposed method is to compute a HZ graph set of the FNN control policy $\mt{\pi}$. For ReLU-activated FNNs, Algorithm 1 in our previous work \cite{zhang2023backward} exactly computes this graph set.
%\XX{Expand the first paragraph. Make better transitions from the previous subsection to this subsection. What does ``Algorithm 1 in \cite{zhang2023backward}'' do?}
%\XX{I moved the following sentences outside this theorem. Please revise it and the theorem. It's weird to have a citation in a theorem statement.}
%Let $\mathcal{H}_\pi = \hz{\pi}$ be the computed graph set of the FNN $\pi$ over the domain $\mathcal{X}$ using Algorithm 1 in \cite{zhang2023backward}, i.e. $\mathcal{H}_\pi = \mathcal{G}_\pi(\mathcal{X})\triangleq \{(\x,\u) \;|\; \u = \mt{\pi}(\x),\x\in\mm{X}\}$. and 
Let $\overline{\mathcal{H}}_f(\mathcal{X},\mathcal{U}) = \hz{f}$ be the computed HZ envelope of the nonlinear function $\bm f$ over the domain $\mathcal{X}\times \mathcal{U}$ using the approach presented in Section \ref{subsec:nonlin}, i.e. $\overline{\mathcal{H}}_f(\mathcal{X},\mathcal{U}) \supseteq \mathcal{G}_f(\mathcal{X},\mathcal{U})$. 

With notations above, the following theorem provides the closed-form identities of an over-approximation of the one-step BRS, $\overline{\mathcal{P}}(\mathcal{T})$, with a given target set represented by an HZ.%, $\mathcal{T}$.

% By combining the techniques presented in Section \ref{subsec:nonlin} and Algorithm 1 in \cite{zhang2023backward}, the following theorem provides the closed-form of an over-approximation of the one-step BRS, $\overline{\mathcal{P}}(\mathcal{T})$ with a given target set represented by an HZ, $\mathcal{T}$.

%\XX{With notations above, the following theorem XXXX}

\begin{theorem}\label{thm:brs}
    Given the state set $\mm{X}$, the input set $\mm{U}$, and a target set represented by an HZ $\mathcal{T}= \hz{\tau} \subset\mathbb{R}^n$, the following HZ is an over-approximated one-step BRS of the NFS \eqref{close-sys}:
    \begin{equation}\label{equ:thm_brs}
        \overline{\mathcal{P}}(\mathcal{T}) = \langle \G^c_p, \G^b_p, \cc_p, \A^c_p, \A^b_p, \bb_p \rangle
    \end{equation}
    where $\cc_p \!=\! \cc_\pi[1:n,:],\; \G^c_p \!=\! \mat{\G^c_\pi[1:n,:] & \bm{0} & \bm 0},\; \G^b_p = \mat{\G^b_\pi[1:n,:] & \bm{0}& \bm 0}$ and
    \begin{align*}
         % \G^c_p &\!=\! \mat{\G^c_\pi[1:n,:] & \bm{0} & \bm 0},\; \G^b_p = \mat{\G^b_\pi[1:n,:] & \bm{0}& \bm 0},\\
         \A^c_p & \!=\! \mat{\A^c_\pi \!&\! \bm{0} \!&\! \bm 0\\ \bm{0} \!&\! \A^c_f \!&\! \bm 0 \\ \bm{0} \!&\! \bm 0 \!&\! \A^c_\tau\\\G^c_\pi \!&\!-\G^c_f[1:n+m,:] \!&\! \bm 0 \\ \bm{0} \!&\! \G^c_f[n+m+1:2n+m,:] \!&\! -\G^c_\tau }\!,\\ 
         \A^b_p &\!=\! \mat{\A^b_\pi \!&\! \bm{0}\!&\! \bm 0\\ \bm{0} \!&\! \A^b_f\!&\! \bm 0 \\ \bm{0} \!&\! \bm 0 \!&\! \A^b_\tau \\\G^b_\pi \!&\!-\G^b_f[1:n+m,:] \!&\! \bm 0 \\ \bm{0} \!&\! \G^b_f[n+m+1:2n+m,:] \!&\! -\G^b_\tau }\!,\\
         \bb_p &\!=\! \mat{\bb_\pi\\\bb_f\\ \bb_\tau\\\cc_f[1:n+m,:]-\cc_\pi \\ \cc_\tau - \cc_f[n+m+1:2n+m,:]}.
    \end{align*}
\end{theorem}
%      By the definition of the one-step BRS, we have $\mathcal{P}(\mathcal{T})  = \{\x \in\mathcal{X}\;|\; \bm g(\x,\u) \in \mathcal{T},\u = \pi(\x)\} = \{\x \;|\; \x \in\mathcal{X}_0, \u \in \mm{U}, \u = \pi(\x), \y = \bm g(\x,\u), \y \in \mathcal{T}\}
%      = \{\x \;|\; [\mt{x}^\top \ \mt{u}^\top]^\top\in \mathcal{H}_\pi, [\mt{x}^\top \ \u^\top\ \mt{y}^\top]^\top\in \mathcal{H}_g(\mathcal{X}_0,\mm{U}), \y\in \mathcal{T}\}.$ 

%      Following the same procedures as in the proof of Theorem \ref{thm:brs}, it can be shown that $\tilde{\mm{P}}(\mm{T}) = \{\x \;|\; [\mt{x}^\top \ \mt{u}^\top]^\top\in \mathcal{H}_\pi, [\mt{x}^\top \ \u^\top\ \mt{y}^\top]^\top\in \overline{\mathcal{H}}_g(\mathcal{X}_0,\mm{U}), \y\in \mathcal{T}\}$. And since $\overline{\mathcal{H}}_g(\mathcal{X}_0,\mm{U}) \supseteq \mathcal{H}_g(\mathcal{X}_0,\mm{U})$, one can conclude that $\tilde{\mm{P}}(\mm{T})$ is an over-approximation of the one-step BRS $\mm{P}(\mm{T})$.
\begin{proof}
    By the definition of the one-step BRS, we have
     $$
     \begin{aligned}
         \mathcal{P}(\mathcal{T}) & = \{\x \in\mathcal{X}\;|\; \bm f(\x,\u) \in \mathcal{T},\u = \mt{\pi}(\x)\} \\
         & = \{\x \;| \mat{\mt{x} \\ \mt{u}}\!\in\!\mathcal{X}\times \mm{U}, \u = \mt{\pi}(\x), \y = \bm f(\x,\u), \y \in \mathcal{T}\}\\
     & = \{\x \;| \mat{\mt{x} \\ \mt{u}} \!\in\! \mathcal{H}_\pi, \mat{\mt{x} \\ \u\\ \mt{y}}\in \mathcal{G}_f(\mathcal{X},\mm{U}), \y\in \mathcal{T}\}.
     \end{aligned}
     $$
     Denote 
     $$\mathcal{H}_p = \{\x \;|\; \mat{\mt{x} \\ \mt{u}}\in \mathcal{H}_\pi, \mat{\mt{x} \\ \u\\ \mt{y}}\in \overline{\mathcal{H}}_f(\mathcal{X},\mm{U}), \y\in \mathcal{T}\}.$$
     Since $\overline{\mathcal{H}}_f(\mathcal{X},\mm{U}) \supseteq \mathcal{G}_f(\mathcal{X},\mm{U})$, it's easy to check that $\mathcal{H}_p$ is an over-approximation of the one-step BRS $\mathcal{P}(\mathcal{T})$, i.e., $\mathcal{H}_p \supseteq \mathcal{P}(\mathcal{T})$. The remaining of this proof is then to show that $\mathcal{H}_p = \overline{\mm{P}}(\mathcal{T})$.
     
     We will first prove that $\mathcal{H}_p\subseteq \overline{\mm{P}}(\mathcal{T})$. Let $\x$ be any element of set $\mathcal{H}_p$. Then, based on the definitions of $\mathcal{H}_\pi$, $\overline{\mathcal{H}}_f(\mathcal{X},\mm{U})$ and $\mathcal{T}$, there exist $\bm \xi^c_\pi, \bm \xi^b_\pi, \bm \xi^c_f, \bm \xi^b_f, \bm \xi^c_\tau$ and $\bm \xi^b_\tau$ such that 
     \begin{align*}
    &(\bxi_\pi^c,\bxi_\pi^b)\in \mathcal{B}(\A^c_\pi,\A^b_\pi,\bb_\pi), \\
     &(\bxi_f^c,\bxi_f^b)\in \mathcal{B}(\A^c_f,\A^b_f,\bb_f), \\
     &(\bxi_\tau^c,\bxi_\tau^b)\in \mathcal{B}(\A^c_\tau,\A^b_\tau,\bb_\tau), 
     \end{align*}
     % $(\bxi_\pi^c,\bxi_\pi^b)\in \mathcal{B}(\A^c_\pi,\A^b_\pi,\bb_\pi)$, 
     % $(\bxi_f^c,\bxi_f^b)\in \mathcal{B}(\A^c_f,\A^b_f,\bb_f)$, 
     % $(\bxi_\tau^c,\bxi_\tau^b)\in \mathcal{B}(\A^c_\tau,\A^b_\tau,\bb_\tau)$, 
     and
     \begin{subequations}\label{equ:allsub}
         \begin{align}
             [\x^\top \; \u^\top]^\top & = \G^c_\pi\bxi_\pi^c+ \G^b_\pi\bxi^b_\pi+\cc_\pi,\label{equ:sub1}\\
             [\x^\top \; \u^\top \; \bm y^\top]^\top &= \G^c_f\bxi_f^c+ \G^b_f\bxi^b_f+\cc_f, \label{equ:sub2}\\
             \y &= \G^c_\tau\bxi_\tau^c+ \G^b_\tau\bxi^b_\tau+\cc_\tau. \label{equ:sub3}
         \end{align}
     \end{subequations}
     %$[\x^\top \; \bm y^\top]^\top = \G^c_f\bxi_f^c+ \G^b_f\bxi^b_f+\cc_f$, $[\x^\top \; \u^\top]^\top = \G^c_\pi\bxi_\pi^c+ \G^b_\pi\bxi^b_\pi+\cc_\pi$ 
     %and $ [\mt{I}_n\; \bm B_d] [\y^\top \; \u^\top]^\top = \G^c_\tau\bxi_\tau^c+ \G^b_\tau\bxi^b_\tau+\cc_\tau$. 
     Therefore, using \eqref{equ:sub1}, we have 
     $$
     \begin{aligned}\x & = [\bm I_n \;\;\bm{0}][\x^\top \;\; \u^\top]^\top \\
     & = [\bm I_n \;\; \bm{0}] (\G^c_\pi\bxi_\pi^c+ \G^b_\pi\bxi^b_\pi+\cc_\pi)\\
     & = \G^c_\pi[1:n,:]\bxi_\pi^c+ \G^b_\pi[1:n,:]\bxi^b_\pi+\cc_\pi[1:n,:].
     \end{aligned}
     $$
     %, and $\u = [\bm 0 \; \bm{I}_m][\x^\top \; \u^\top]^\top = \G^c_\pi[n+1:n+m,:]\bxi_\pi^c+ \G^b_\pi[n+1:n+m,:]\bxi^b_\pi+\cc_\pi[n+1:n+m,:]$.
    Using \eqref{equ:sub2}, we have 
     \begin{align*}
     [\x^\top \; \u^\top]^\top &=  [\bm I_{n+m} \;\; \bm{0}][\x^\top \; \u^\top\; \y^\top]^\top \\
     &=  [\bm I_{n+m} \;\; \bm{0}] (\G^c_f\bxi_f^c+ \G^b_f\bxi^b_f+\cc_f)\\
     & =  \G^c_f[1:n+m,:]\bxi_f^c+ \G^b_f[1:n+m,:]\bxi^b_f\\
      &\quad  +\cc_f[1:n+m,:]    
     \end{align*}
     and 
     $$
     \begin{aligned}
     \y & =  [\bm 0 \;\; \bm 0 \;\;\bm{I}_n][\x^\top\;\; \u^\top\;\; \y^\top]^\top \\ 
     & =  \G^c_f[n\!+\!m\!+\!1\!:\!2n\!+\!m,:]\bxi_f^c \!+\! \G^b_f[n\!+\!m\!+\!1\!:\!2n\!+\!m,:]\bxi^b_f\\
     & \quad+\cc_f[n\!+\!m\!+\!1\!:\!2n\!+\!m,:].
     \end{aligned}
     $$ 
     Combining \eqref{equ:sub1} and $[\x^\top \; \u^\top]^\top$ from \eqref{equ:sub2}, we have 
     $$
     \begin{aligned}
         \G^c_\pi\bxi_\pi^c\!+\!\G^b_\pi\bxi^b_\pi\!+\!\cc_\pi& =  \G^c_f[1\!:\!n\!+\!m,:]\bxi_f^c \!+\! \G^b_f[1\!:\!n\!+\!m,:]\bxi^b_f\\ & \quad +\cc_f[1\!:\!n\!+\!m,:].
     \end{aligned}
     $$
     Similarly, by combing \eqref{equ:sub3} and $\y$ computed from \eqref{equ:sub2}, we get 
     $$
     \begin{aligned}
         & \G^c_\tau\bxi_\tau^c+ \G^b_\tau\bxi^b_\tau+\cc_\tau =  \G^c_f[n\!+\!m\!+\!1\!:\!2n\!+\!m,:]\bxi_f^c \\ 
         & \quad  + \G^b_f[n\!+\!m\!+\!1\!:\!2n\!+\!m,:]\bxi^b_f + \cc_f[n\!+\!m\!+\!1\!:\!2n\!+\!m,:].
     \end{aligned}
     $$
     Let $$\bxi^c = \mat{\bxi^{c}_\pi\\\bxi^{c}_f\\\bxi^c_\tau} \text{ and } \bxi^b = \mat{\bxi^b_\pi\\\bxi^b_f\\\bxi^b_\tau}.$$ 
     Then, it's easy to check that $(\bxi^c,\bxi^b)\in\mathcal{B}(\A^c_p,\A^b_p,\bb_p)$ and 
     \begin{align*}
     \G^c_p\bxi^c+ \G^b_p\bxi^b+\cc_p  & = [\G^c_\pi[1\!:\!n,:] \;\; \bm{0}\;\; \bm{0}]\bxi^c\\
     & + [\G^b_\pi[1\!:\!n,:] \;\; \bm{0} \;\; \bm{0}]\bxi^b +\cc_\pi[1\!:\!n,:].
     \end{align*}
     Thus, we have $\x\in\overline{\mm{P}}(\mathcal{T})$. Since $\x$ is arbitrary, we get $\mathcal{H}_p\subseteq \overline{\mm{P}}(\mathcal{T})$. 

     Next, we'll show that $\mathcal{H}_p\supseteq \overline{\mm{P}}(\mathcal{T})$. Let $\x\in\overline{\mm{P}}(\mathcal{T})$. Then, there exist $(\bxi^c,\bxi^b) \in\mathcal{B}(\A^c_p,\A^b_p,\bb_p)$ such that $\x = \G^c_p\bxi^c+ \G^b_p\bxi^b+\cc_p.$ 
     
     Partitioning $\bxi^c$ as $\bxi^c = [(\bxi^{c}_\pi)^\top \; (\bxi^{c}_f)^\top \;(\bxi^c_\tau)^\top]^\top$ and $\bxi^b$ as $\bxi^b = [(\bxi^b_\pi)^\top\;(\bxi^b_f)^\top\;(\bxi^b_\tau)^\top]^\top$, it follows that 
     % $(\bxi_\pi^c,\bxi_\pi^b)\in \mathcal{B}(\A^c_\pi,\A^b_\pi,\bb_\pi)$, $(\bxi_f^c,\bxi_f^b)\in \mathcal{B}(\A^c_f,\A^b_f,\bb_f)$, $(\bxi_\tau^c,\bxi_\tau^b)\in \mathcal{B}(\A^c_\tau,\A^b_\tau,\bb_\tau)$ 
     \begin{align*}
    &(\bxi_\pi^c,\bxi_\pi^b)\in \mathcal{B}(\A^c_\pi,\A^b_\pi,\bb_\pi),\\ &(\bxi_f^c,\bxi_f^b)\in \mathcal{B}(\A^c_f,\A^b_f,\bb_f),\\ &(\bxi_\tau^c,\bxi_\tau^b)\in \mathcal{B}(\A^c_\tau,\A^b_\tau,\bb_\tau),
     \end{align*}
     and 
     \begin{align*}
         \x  & = \G^c_\pi[1:n,:]\bxi_\pi^c+ \G^b_\pi[1:n,:]\bxi^b_\pi+\cc_\pi[1:n,:]\\
         & = \G^c_f[1:n,:]\bxi_f^c+ \G^b_f[1:n,:]\bxi^b_f+\cc_f[1:n,:].
     \end{align*}
     Choose 
     \begin{align*}
         \u &=  \G^c_\pi[n+1\!:\!n+m,:]\bxi_\pi^c+ \G^b_\pi[n+1\!:\!n+m,:]\bxi^b_\pi\\
         &\quad +\cc_\pi[n+1\!:\!n+m,:],\\
         \y &=   \G^c_f[n\!+\!m\!+\!1\!:\!2n\!+\!m,:]\bxi_f^c\!+\!\G^b_f[n\!+\!m\!+\!1\!:\!2n\!+\!m,:]\bxi^b_f\\
         & \quad+\cc_f[n\!+\!m\!+\!1\!:\!2n\!+\!m,:].
     \end{align*}
     
     Then, it's easy to check that $\x$,$\u$ and $\y$ satisfy \eqref{equ:allsub}. Thus, $\x\in \mathcal{H}_p$. Since $\x$ is arbitrary, $\mathcal{H}_p\supseteq \overline{\mm{P}}(\mathcal{T})$. 
     
     Therefore, we conclude that $\mathcal{H}_p = \overline{\mm{P}}(\mathcal{T})$. %This completes the proof.
\end{proof}

%Theorem \ref{thm:brs} provides a closed-form HZ representation for the one-step BRS of the neural feedback system \eqref{close-sys} which has general nonlinear dynamics. 
Theorem \ref{thm:brs} can be applied to NFSs with general nonlinear plant models, but the set $\overline{\mathcal{P}}(\mathcal{T})$ computed may induce redundant equality constraints in the HZ representation. % for systems with simpler nonlinear structures. 
The following result provides a BRS representation of $\overline{\mathcal{P}}(\mathcal{T})$ with less redundant constraints for NFSs with a special class of nonlinear plant models.

\begin{corollary}\label{coro:linear_input}
Consider an NFS \eqref{close-sys} where $\f(\x,\u) = \g(\x)+\bm B_d \u$. Given the state set $\mm{X}$, the input set $\mm{U}$ and a target set represented by an HZ $\mathcal{T}= \hz{\tau} \subset\mathbb{R}^n$, let $\mathcal{H}_\pi = \hz{\pi}$ be the computed HZ-based graph set of the FNN $\pi$ over the domain $\mathcal{X}$. 
%that contains the nonlinear dynamics and a linear input term: $\x(t+1) = \f(\x(t),\u(t)) = \g(\x(t))+\bm B_d \u(t) = \g(\x(t))+\bm B_d \pi(\x(t))$. Given the state set $\mm{X}$, the input set $\mm{U}$ and a target set represented by an HZ $\mathcal{T}= \hz{\tau} \subset\mathbb{R}^n$, let $\mathcal{H}_\pi = \hz{\pi}$ be the computed HZ-based graph set of the FNN $\pi$ over the domain $\mathcal{X}$. 
%Let sets $\mm{X}$, $\mm{U}$, $\mm{T}$ and $\mm{H}_\pi$ be the same as defined in Theorem \ref{thm:brs}\XX{avoid statements like ``as defined in Theorem \ref{thm:brs}''. Define it or  use labling to refer to it.}. 
Let $\overline{\mathcal{H}}_g(\mathcal{X}) = \hz{g}$ be the computed HZ envelope of the nonlinear function $\bm g$ over the domain $\mathcal{X}$. 
%using the approach presented in Section \ref{subsec:nonlin} \XX{don't use statements like ``presented in Section \ref{subsec:nonlin}''. I never saw a theorem statement like this.}. 
Then, an over-approximated one-step BRS of the NFS can be computed as an HZ given by
$
        \tilde{\mathcal{P}}(\mathcal{T}) = \langle \G^c_p, \G^b_p, \cc_p, {\tilde \A^c_{p}}, { \tilde \A^b_{p}}, {\tilde \bb_{p}} \rangle,
$ 
    where $ \G^c_p, \G^b_p, \cc_p$ are the same as in Theorem \ref{thm:brs}, 
    % $\tilde{\bb}_{p} = \mat{\bb_\pi^\top & \bb_g^\top & \bb_\tau^\top & \bb_1^\top & \bb_2^\top}$, $\tilde{\A}^c_{p} = \mat{\operatorname{diag}(\A^c_\pi,\A^c_g,\A^c_\tau) \\ \A_1 }$, $\tilde{\A}^b_{p} = \mat{\operatorname{diag}(\A^b_\pi,\A^b_g,\A^b_\tau) \\ \A_2 }$ with $ \bb_1 = \cc_g[1:n,:]-\cc_\pi[1:n,:]$, $ \bb_2 =\cc_\tau- \bm{B}_d \cc_\pi[n+1:n+m,:] - \cc_g[n+1:2n,:]$, 
    \begin{align*}
    \tilde{\bb}_{p} &= \mat{\bb_\pi^\top & \bb_g^\top & \bb_\tau^\top & \bb_1^\top & \bb_2^\top},\\ \tilde{\A}^c_{p} &= \mat{\operatorname{diag}(\A^c_\pi,\A^c_g,\A^c_\tau) \\ \A_1 },\\ 
    \tilde{\A}^b_{p} &= \mat{\operatorname{diag}(\A^b_\pi,\A^b_g,\A^b_\tau) \\ \A_2 },
    \end{align*}
    with 
    \begin{align*}
    \bb_1 &= \cc_g[1:n,:]-\cc_\pi[1:n,:],\\
    \bb_2 &=\cc_\tau- \bm{B}_d \cc_\pi[n+1:n+m,:] - \cc_g[n+1:2n,:],\\
    %\end{align*}
    %$ \bb_1 = \cc_g[1:n,:]-\cc_\pi[1:n,:]$, $ \bb_2 =\cc_\tau- \bm{B}_d \cc_\pi[n+1:n+m,:] - \cc_g[n+1:2n,:]$, 
    %and %\\\cc_g[1:n,:]-\cc_\pi[1:n,:] \\ \cc_\tau- \bm{B}_d \cc_\pi[n+1:n+m,:] - \cc_g[n+1:2n,:]}$\XX{simplify}
    %\begin{align*}
    \A_1 & \!=\! \mat{\G^c_\pi[1:n,:] \!&\!-\G^c_g[1:n,:] \!&\! \bm 0 \\ \bm{B}_d \G^c_\pi[n+1:n+m,:] \!&\! \G^c_g[n+1:2n,:] \!&\! -\G^c_\tau},\\ 
    \A_2 &\!=\! \mat{\G^b_\pi[1:n,:] \!&\!-\G^b_g[1:n,:] \!&\! \bm 0 \\ \bm{B}_d \G^b_\pi[n+1:n+m,:] \!&\! \G^b_g[n+1:2n,:] \!&\! -\G^b_\tau \!}.
    \end{align*}
\end{corollary}

The proof of Corollary \ref{coro:linear_input} follows the same procedures as the proof of Theorem \ref{thm:brs}, and is thus omitted.% due to space limitation.

% \begin{remark}
%     Theorem \ref{thm:brs} can be applied to various commonly used nonlinear feedback systems. One example is the linear input term case where $\f(\x,\u) = \g(\x)+\bm B_d \u$. In this case, if we have $\overline{\mathcal{H}}_f(\mathcal{X}) = \hz{f}$, then $\overline{\mathcal{H}}_g(\mathcal{X},\mathcal{U}) = \hz{g}$, where
%     \begin{align*}
%         & \G^c_g = \mat{\bm 0 & \G^c_f[1:n,:]&\bm 0\\ \G^c_\pi[n+1:n+m,:]& \bm 0 & \bm 0\\ \bm B_d \G^c_\pi[n+1:n+m,:] & \G^c_f[n+1:2n,:] & \bm 0},\\
%         & \G^b_g = \mat{\bm 0 & \G^b_f[1:n,:]&\bm 0\\ \G^b_\pi[n+1:n+m,:]& \bm 0 & \bm 0\\ \bm B_d \G^b_\pi[n+1:n+m,:] & \G^b_f[n+1:2n,:] & \bm 0},\\
%         & \A^c_g = diag(\A^c_\pi,\A^c_f,\A^c_\tau),\;\A^b_g = diag(\A^b_\pi,\A^b_f,\A^b_\tau),\\
%         & \cc_g = \mat{\cc_f[1:n,:]\\ \cc_\pi[n+1:n+m,:]\\ \bm B_d\cc_\pi[n+1:n+m,:]+\cc_f[n+1:2n,:]},\\
%         & \bb_g = \mat{\bb_\pi^\top & \bb_f^\top &\bb_\tau^\top}^\top.
%     \end{align*}
%     Then, it's easy to check that $\tilde{\mathcal{P}}(\mm{T}) = \overline{\mm{P}}(\mm{T})$.
% \end{remark}

% \subsection{Algorithm for Over-Approximating K-Step BRSs}
% \subsection{Over-approximation of k-step BRS}
%\XX{better subsection title}
%We can iteratively compute the over-approximations of BRSs for multiple steps by applying Theorem \ref{thm:brs} to \eqref{equ:t-step}. 
%Based on Theorem \ref{thm:brs} to \eqref{equ:t-step}, 
\subsection{Refinement of Over-approximated Multi-step BRS}
Based on the results of the preceding subsections, 
the $T$-step over-approximations of BRS for \eqref{close-sys} can be computed iteratively as follows:
\begin{equation}\label{equ:t-step}
    \begin{aligned}
    \overline{\mathcal{P}}_0(\mathcal{T}) =  \mathcal{T},\;\overline{\mathcal{P}}_t(\mathcal{T}) 
    %= &\{x \;|\; A_d \x+B_d\pi(\x)\in \mathcal{P}_{t-1}(\mathcal{T}), \mat{\x\\\u} \in \mathcal{F}\}\\
    = \overline{\mathcal{P}}(\overline{\mathcal{P}}_{t-1}(\mathcal{T})),\; t=1,\dots,T.\\
\end{aligned}
\end{equation}
Since Theorem \ref{thm:brs} computes $\overline{\mm{H}}_f(\mm{X},\mm{U})$ over the entire state-input domain $\mm{X} \times \mm{U}$, a naive application of \eqref{equ:t-step} based on Theorem \ref{thm:brs} may lead to conservative approximations for multi-step computations. %over time horizons. %\XX{my understanding is that ``wrapping effect'' is something related to interval analysis. What do you mean here?}\hang{Yes wrapping effect is not accurate. What I am trying to say is that we are using the over-approximation of BRS at time $k$ to compute the BRS at time $k+1$, which accumulates the approximation error.}. 
In order to mitigate the conservativeness, we introduce a refinement procedure aimed at computing $\overline{\mm{H}}_f$ within more tightly constrained prior sets, denoted as $\mm{X}^p \times \mm{U}^p$ where $p$ stands for ``prior''. Such procedure allows the nonlinearity to be approximated over tighter sets, thus leading to less conservative results. Concretely, the refinement procedure employs a recursive approach that utilizes the BRSs obtained from the previous $\overline{\mm{H}}_f$ to compute the new prior sets. 

With the same notations as in Theorem \ref{thm:brs}, the following lemma provides a closed-form expression of an over-approximation of BRS $\overline{\mm{P}}(\mm{T})$ over the prior sets $\mm{X}^p \times \mm{U}^p$. 

\begin{lemma} \label{coro:refine}
Given the state set $\mm{X}$, the input set $\mm{U}$, the target set $\mm{T}$, and the graph set of $\mt{\pi}$, $\mathcal{H}_\pi$, suppose that the prior sets $\mm{X}^p \times \mm{U}^p$ satisfies
    \begin{align} \label{eq:prior_sets}
         \mm{H}_\pi \cap_{\mat{ \mt{I}_n & \mt{0}}}\mm{P}(\mm{T}) \subseteq \mm{X}^p \times \mm{U}^p \subseteq \mm{X} \times \mm{U}.
    \end{align}
    Let  $\overline{\mathcal{H}}_f(\mathcal{X}^p,\mathcal{U}^p) = \hz{fp}$ be the over-approximated graph set of the nonlinear function $\bm f$ over the domain $\mathcal{X}^p\times \mathcal{U}^p$. Then, the one-step BRS of the NFS \eqref{close-sys} can be over-approximated by an HZ whose \emph{HCG}-representation is the same as \eqref{equ:thm_brs} by replacing $\hz{f}$ with $\hz{fp}$.
\end{lemma}
\begin{proof}
    From \eqref{eq:prior_sets}, we know that $\mm{P}(\mm{T}) \subseteq \mm{X}^p \subseteq \mm{X}$ and $ \mm{U}_{\mm{T}} \subseteq \mm{U}^p \subseteq \mm{U}$, where $\mm{U}_{\mm{T}}$ denotes the projection of $\mm{H}_\pi \cap_{\mat{ \mt{I}_n & \mt{0}}}\mm{P}(\mm{T})$ onto the $\mt{u}$ space.
    By the definition of the one-step BRS, we have 
    \begin{align*}
           \mathcal{P}(\mathcal{T})  &= \{\x \in\mathcal{X}\;|\; \bm f(\x,\u) \in \mathcal{T},\u = \mt{\pi}(\x)\} \\
           &= \{ \mt{x} \in \mm{X}^p \mid  \mt{f}(\mt{x}, \mt{u}) \in \mm{T}, \mt{u}= \mt{\pi}(\mt{x})\} \\ 
           &= \{\x \;|\;  [\mt{x}^\top \ \u^\top\ \mt{y}^\top]^\top\in \mathcal{G}_f(\mathcal{X}^p,\mm{U}^p),\\
           & \ \quad \qquad \mt{u}= \mt{\pi}(\mt{x}), \y\in \mathcal{T}\} \\
           &=\{\x \;|\; [\mt{x}^\top \ \mt{u}^\top]^\top\in \mm{H}_\pi \cap_{\mat{ \mt{I}_n & \mt{0}}}\mm{P}(\mm{T}), \\
           & \ \quad \qquad [\mt{x}^\top \ \u^\top\ \mt{y}^\top]^\top\in \mathcal{G}_f(\mathcal{X}^p,\mm{U}^p), \y\in \mathcal{T}\} \\
           &= \{\x \;|\; [\mt{x}^\top \ \u^\top\ \mt{y}^\top]^\top\in \mathcal{G}_f(\mathcal{X}^p,\mm{U}^p), \\
           & \ \quad \qquad [\mt{x}^\top \ \mt{u}^\top]^\top\in \mm{H}_\pi, \y\in \mathcal{T}\}. 
    \end{align*}
    The remaining proof is the same as that of Theorem \ref{thm:brs}.
\end{proof}

\setlength{\textfloatsep}{8pt}
\begin{algorithm}[h]
    \SetNoFillComment
    \caption{BReachNonlin-HZ}\label{alg:brs_comp}
    \KwIn{State domain $\mathcal{X}$, control input domain $\mm{U}$, target set $\mm{T}$, nonlinear plant model $\mt{f}$, neural network controller $\mt{\pi}$, large scalars $\alpha,\beta >0$, time horizon $T$, the number of refinement epochs $n_r$.}
    \KwOut{Refined over-approximations of $t$-step BRSs from $t=1$ to $T$ $\overline{\mm{P}}_{1:T}(\mm{T})$.}
    \textbf{Initialization}: $\overline{\mm{P}}_0(\mm{T}) \leftarrow \mm{T}$; $\mm{X}^p_{1:T} \leftarrow \mm{X}$;\ $\mm{U}^p_{1:T} \leftarrow \mm{U}$; \label{alg:init}\\

    Compute $\mathcal{H}_\pi$ with $\alpha$ and $\beta$ \tcp*{Alg 1 in \cite{zhang2023backward}} \label{alg:hpi}
    
    % $[\overline{\mm{P}}_{1:K}(\mm{T}), \mm{H}_\pi] \leftarrow {BReachNonlin\text{-}HZ}(\mm{X}, \mm{T}, \mt{f}, \mt{\pi}, \alpha, \beta, K)$ \\
           
        \For{$r \in \{0,1,\dots,n_r$\}}{
     
        \For{$t \in \{1,2,\dots,T$\}}{ 
        Compute $\overline{\mathcal{H}}_f(\mm{X}^p_t, \mm{U}^p_t)$\tcp*{Section \ref{subsec:nonlin}} \label{alg:hf}
        
        $\overline{\mm{P}}_t(\mm{T}) \leftarrow \text{Pre}(\overline{\mm{P}}_{t-1}(\mm{T}), \overline{\mm{H}}_f(\mm{X}^p_t, \mm{U}^p_t), \mm{H}_\pi)$\tcp*{Over-approximate one-step BRS} \label{alg:pre}

        \If{$n_r==0$ \rm \textbf{or} $r==n_r$}{\textbf{continue};\\ }
        
        $\mm{X}^p_t$, $\mm{U}^p_t$ $\leftarrow$ ${GetBox}(\overline{\mm{P}}_t(\mm{T}), \mm{H}_\pi)$\tcp*{The prior sets for $\mm{P}_t(\mm{T})$.} \label{alg:getbox}
        
        % $\overline{\mathcal{H}}_f(\mm{X}_0)$ $\leftarrow$ compute the over-approximation of nonlinear dynamics;\\
        }
    }
    % \Return{$\overline{\mm{P}}_1(\mm{T}), \overline{\mm{P}}_2(\mm{T}), \cdots, \overline{\mm{P}}_K(\mm{T})$}
        \Return{$\overline{\mm{P}}_{1:T}(\mm{T})$}
\end{algorithm}

%\XX{Specify the lines when explaining Algorithm 1 below. For example, Line 1 is XXX} \hang{Done. I used brackets to specify each line.}
Algorithm \ref{alg:brs_comp} summarizes the over-approximation of BRS for the closed-loop system \eqref{close-sys} with the refinement procedure. The prior sets $\mm{X}^p \times \mm{U}^p$ for all the time horizons are initialized as $\mm{X} \times \mm{U}$ (line \ref{alg:init}) and $\mm{H}_\pi$ is computed over $\mm{X}$ with the large scalars $\alpha$ and $\beta$ by using Algorithm 1 in \cite{zhang2023backward} (line \ref{alg:hpi}). For each time horizon $t$, we compute the over-approximation of the nonlinear plant model over the prior set $\mm{X}^p_t$ and $\mm{U}^p_t$ (line \ref{alg:hf}) and then the over-approximated one-step BRS $\overline{\mm{P}}_t(\mm{T})$ is computed by Lemma \ref{coro:refine} (line \ref{alg:pre}). If $n_r=0$ or the refinement epoch reaches the maximum number $n_r$, the refinement procedure will be skipped and the computed BRSs $\overline{\mm{P}}_{t}(\mm{T})$ for $t=1$ to $T$ will be returned; otherwise, the tighter prior sets are computed based on the over-approximated BRSs and then passed to the next refinement epoch (line \ref{alg:getbox}). %\XX{move this to algorithm 1?} \hang{Done.}
    Concretely, in line \ref{alg:getbox}, the bounding box of $\mm{H}_\pi^{\mm{T}} \triangleq \mm{H}_\pi \cap_{\mat{ \mt{I}_n & \mt{0}}}\overline{\mm{P}}_t(\mm{T})$, is computed by the $GetBox$ and is set to be the prior sets for the refinement procedure. The $GetBox$ only shrinks the bounding intervals of $\mm{X} \times \mm{U}$ over the dimensions contributing to the nonlinearity. Let $\mt{r} = [\mt{x}^\top \ \mt{u}^\top]^\top \in \mm{H}_\pi$. Such shrinking interval, denoted as $\li r_f^{lb}, r_f^{ub} \ri$, can be obtained by solving two Mixed-Integer linear Programs (MILPs): $r_f^{lb} = \min_{\mt{r} \in \mm{H}_\pi^{\mm{T}}}  r_f$, $r_f^{ub} = \max_{\mt{r}\in \mm{H}_\pi^{\mm{T}}} r_f $. % as follows:

\begin{remark}
    The number of refinement epochs $n_r$ in Algorithm \ref{alg:brs_comp} represents the trade-off between the set approximation accuracy and the computation efficiency. Specifically, the refinement procedure mitigates the conservativeness of the set approximation but requires computing $\overline{\mm{H}}_f(\mm{X}_t^p,\mm{U}_t^p)$ and the bounding hyper-boxes of HZs at each epoch, which introduces extra computation time. Simulation results for different numbers of epochs are given in Section \ref{sec:ex}.
\end{remark}

\begin{remark}
The over-approximated BRSs can be used for the safety verification of NFSs \cite{bird2023hybrid,zhang2023backward}. Consider an HZ initial state set $\mm{X}_0 \subset \mm{X}$ and an HZ unsafe region $\mm{O} \subset \mm{X}$. Suppose that the $t$-step BRS of $\mm{O}$ is over-approximated as $\overline{\mm{P}}_t(\mm{O})$ via Algorithm \ref{alg:brs_comp} for $t=1,\cdots,T$. 
%with arbitrary positive integer $K$. 
Then any trajectory that starts from $\mm{X}_0$ will not enter into the unsafe region $\mm{O}$ within $T$ time steps if all the over-approximated BRSs $\overline{\mm{P}}_{1:T}(\mm{O})$ have no intersections with the initial set $\mm{X}_0$. By \cite[Proposition 7]{bird2023hybrid} and \cite[Lemma 1]{zhang2023backward}, verifying the emptiness of the intersection of $\mathcal{X}_0$ and $\overline{\mathcal{P}}_t(\mm{O})$ is equivalent to solve an MILP. 

% Note that checking the intersections is not a necessary condition for safety verification since the computed BRSs are over-approximated.

%\XX{Comments on the usefulness of over-approximated BRS. Inner-approximated BRS? Say something about safety verification. }
\end{remark}

% \setlength{\textfloatsep}{8pt}
% \begin{algorithm}[t]
%     \SetNoFillComment
%     \caption{BReachNonlin-HZ}\label{alg:1}
%     \KwIn{HZ domain $\mathcal{X}$, target set $\mm{T}$, nonlinear dynamics $\mt{f}$, neural network controller $\mt{\pi}$, large scalars $\alpha,\beta >0$, time horizon $K$.}
%     \KwOut{Over-approximations of BRSs $\overline{\mm{P}}_{1:K}(\mm{T})$.}
%     \textbf{Initialization}: $\overline{\mm{P}}_0(\mm{T}) \leftarrow \mm{T}$ \\
%     $\overline{\mathcal{H}}_f(\mm{X})$ $\leftarrow$ compute the over-approximation of nonlinear dynamics;\\
%     $\mathcal{H}_\pi$ $\leftarrow ExactFNN(\mt{\pi}, \alpha,\beta)$\tcp*{Alg 1 in \cite{zhang2023backward}}
   
%     \For{$k \in \{1,2,\dots,K$\}}{
%         $\overline{\mm{P}}_k(\mm{T}) \leftarrow {BackReach}(\mm{P}_{k-1}(\mm{T}), \overline{\mm{H}}_f(\mm{X}), \mm{H}_\pi)$\tcp*{Eq. \eqref{equ:thm2}}
%     }
%     \Return{$\overline{\mm{P}}_1(\mm{T}), \overline{\mm{P}}_2(\mm{T}), \cdots, \overline{\mm{P}}_K(\mm{T})$}
% \end{algorithm}

% \section{Backward Reachability Analysis of Neural Feedback Loop With Other Activation Functions}
\section{Neural Feedback System With General Types of Activation Functions} \label{sec:other_activate}
%\XX{Need better title. ``Other Types '' seems weird to me.}
%\yuhao{"Beyond ReLU-activated Neural Feedback System"?}
%\XX{Can we separate two subsections (i) piecewise activation function - this should be relatively trivial without involving OVERT/SOS, and we can leverage our response to LCSS paper; (ii) non-piecewise activation function}\hang{Sure, we can also put one lemma and two theorems in this part.}
In this section, we extend the results of the preceding section that considers ReLU-activated FNNs to FNNs with more general activation functions.  We categorize the activation functions as either piecewise linear or non-piecewise linear. For FNNs with piecewise linear activation functions, their \emph{exact} graph sets can be constructed, in which case Theorem \ref{thm:brs} and Lemma \ref{coro:refine} are directly applicable. For FNNs with non-piecewise linear activation functions, their graph sets can be over-approximated by using the methods in Section \ref{subsec:nonlin}, based on which the over-approximated BRSs can be computed.

\subsection{Piecewise Linear Activation Functions} \label{subsec:prelu}
%\XX{add a comment that our method can deal with XXX} 
The graph of piecewise linear activation functions can be exactly represented by a finite union of polytopes that is equivalent to an HZ. For simplicity, we use the Parametric ReLU (PReLU) activation function as an example to show the construction of the graph set.

The PReLU is a parametric activation function defined as 
$$
\sigma\left(x\right)=\left\{\begin{array}{ll}
x, & \text { if } x>0, \\
a x, & \text { if } x \leq 0.
\end{array} \right.
$$
Here $a\in \mathbb{R}_{\geq 0}$ is a parameter that is learned along with other neural network parameters \cite{he2015delving}. When $a = 0$, it becomes the conventional ReLU; when $a$ is a small positive number, it becomes the Leaky ReLU. 
Similar to the construction of HZ for ReLU in \cite{zhang2023backward}, given an interval domain $\li-\alpha,\beta\ri$, the two line segments of PReLU can be exactly represented as two HZs: % $\mathcal{H}_{1,a}$ and $\mathcal{H}_2$ given as follows:
\begin{align*}
    \mathcal{H}_{1,a} &= \left\langle\!  \mat{\frac{\alpha}{2}\\\frac{\alpha\cdot a}{2}},\emptyset, \mat{\frac{-\alpha}{2}\\\frac{-\alpha\cdot a}{2}},\emptyset,\emptyset,\emptyset \right\rangle, \\
    \mathcal{H}_2 &=  \left\langle  \mat{\frac{\beta}{2}\\\frac{\beta}{2}},\emptyset,\mat{\frac{\beta}{2}\\\frac{\beta}{2}},\emptyset,\emptyset,\emptyset \right\rangle.
\end{align*}
% $
%     \mathcal{H}_{1,a} \!=\! \left\langle\!  \mat{\frac{\alpha}{2} & \frac{\alpha\cdot a}{2}}^\top,\emptyset, \mat{\frac{-\alpha}{2} & \frac{-\alpha\cdot a}{2}}^\top,\emptyset,\emptyset,\emptyset \!\right\rangle, 
%     \mathcal{H}_2 \!=\!  \left\langle\!  \mat{\frac{\beta}{2}&\frac{\beta}{2}}^\top,\emptyset,\mat{\frac{\beta}{2}&\frac{\beta}{2}}^\top,\emptyset,\emptyset,\emptyset \!\right\rangle\!$. 
The union of $\mathcal{H}_{1,a}$ and $\mathcal{H}_2$ can be directly computed as a single HZ, $\mathcal{H}_{\sigma}$, using Proposition 1 in \cite{bird2021unions}. Applying Proposition 3 in \cite{zhang2022reachability} we can remove the redundant continuous generators to get
\begin{equation}\label{equ:H_prelu}
    \mathcal{H}_{\sigma} = \mathcal{H}_{1,a} \cup \mathcal{H}_2 = \langle \G^c_{h,a},\allowbreak \G^b_{h,a},\allowbreak \cc_{h},\allowbreak \A^c_{h},\allowbreak \A^b_{h},\allowbreak \bb_{h} \rangle
\end{equation}
where
\begin{align*}
    &  \G^c_{h,a}   =  \mat{-\frac{\alpha}{2} & -\frac{\beta}{2} & 0& 0\\ -\frac{\alpha\cdot a}{2} & -\frac{\beta}{2} & 0& 0}, \G^b_{h,a}  =  \mat{-\frac{\alpha}{2}\\-\frac{\alpha\cdot a}{2}}, \cc_h  =  \mat{\frac{\beta}{2} \\ \frac{\beta}{2}},\\
    & \A^c_h  =\mat{1& 0 & 1 & 0\\0 & 1 & 0 & 1}, \A^b_h  = \mat{1\\-1}, \bb_h  =\mat{1\\1}.
\end{align*}

%Note that when $a = 0$, equation \eqref{equ:H_prelu} becomes identical to the HZ representation of ReLU.
Since \eqref{equ:H_prelu} exactly represents PReLU over an interval domain, Algorithm 1 in \cite{zhang2023backward} can be directly applied to PReLU-activated FNNs, and based on that, the exact graph set of FNNs $\mm{H}_\pi$ can be computed. Hence, Theorem \ref{thm:brs}, Lemma \ref{coro:refine}, and Algorithm \ref{alg:brs_comp} can be directly employed to over-approximate BRSs for a PReLU-activated NFS \eqref{close-sys}.

\subsection{Non-Piecewise Linear Activation Functions} \label{subsec:nonlin_act}
%\XX{Need better title. More formal categorization than ``Non-Piecewise Linear Activation Functions''.} %\yuhao{Maybe "smooth activation functions"?}
In addition to piecwise linear activation functions, non-piecewise linear activation functions such as sigmoid and tanh functions, are also widely employed in NNs \cite{DUBEY202292}. Since the graph of non-piecewise linear functions cannot be represented exactly by HZs, we will construct an over-approximation of the graph set of FNNs $\mm{H}_\pi$.
%denoted as $\overline{\mm{H}}_\pi$, such that $\overline{\mm{H}}_\pi \supset \mm{H}_\pi$.

Since an activation function is a scalar function, we use the methods discussed in Section \ref{subsec:nonlin} to construct an HZ $\overline{\mm{H}}_a$ over the interval domain $\li -\alpha, \beta \ri$ such that 
\begin{align} \label{eq:graph_nonlin_act}
    \overline{\mm{H}}_a \supset \mm{H}_a = \left\{ [x \ y]^\top \mid y = \sigma(x), x \in \li -\alpha, \beta \ri\right\}.
\end{align}
The following lemma gives the HZ over-approximation of the graph of vector-valued activation function $\mt{\phi}$ that is constructed by component-wise repetition of the scalar activation function $\sigma$ over a domain $\mm{X}_a$ represented as an HZ.
\begin{lemma} \label{lem:h_pi_over}
    Assume that a domain $\mm{X}_a$ is represented as an HZ such that $\mm{X}_a \subseteq \li -\alpha \mt{1},\beta \mt{1}\ri \subset \mathbb{R}^{n_k}$. Let $\mm{G}(\mt{\phi},\mm{X}_a)$ be the graph of the vector-valued activation function $\mt{\phi}:\mathbb{R}^{n_k} \rightarrow \mathbb{R}^{n_k}$, i.e., $\mm{G}(\mt{\phi},\mm{X}_a) = \left \{ [\mt{x}^\top \ \mt{y}^\top]^\top \mid \mt{y} = \mt{\phi}(\mt{x}), \mt{x} \in \mm{X}_a \right\}$. Then $\mm{G}(\mt{\phi}, \mm{X}_a)$ can be over-approximated by an HZ, denoted as $\overline{\mm{G}}(\mt{\phi}, \mm{X}_a)$,  where
    \begin{align} \label{eq:h_pi_over}
        \overline{\mm{G}}(\mt{\phi}, \mm{X}_a) = \left( \mt{P} \overline{\mm{H}}_a^{n_k} \right)\cap_{\mat{\mt{I}_{n_k} & \mt{0}}} \mm{X}_a,
    \end{align}
    $\bm{P}=[\bm{e}_{2}\; \bm{e}_{4}\;\cdots\; \bm{e}_{2n_{k}} \; \bm{e}_{1}\; \bm{e}_{3}\;\cdots\; \bm{e}_{2n_{k}-1}]^T\in \mathbb{R}^{2n_{k}\times 2n_{k}}$ is a permutation matrix, and $\overline{\mm{H}}_{a}$ is given in \eqref{eq:graph_nonlin_act}.
\end{lemma}
\begin{proof}
    For simplicity, denote $\mm{I}$ as $\li -\alpha \mt{1}, \beta \mt{1}\ri$. Since $\mm{X}_a \subseteq \mm{I}$, we have $\mm{G}(\mt{\phi}, \mm{X}_a) = \mm{G}(\mt{\phi}, \mm{I}) \cap_{\mat{\mt{I}_{n_k} & \mt{0}}}\mm{X}_a$. Then the proof of \eqref{eq:h_pi_over} is equivalent to showing that $\mm{G}(\mt{\phi}, \mm{I}) \subseteq  \mt{P} \overline{\mm{H}}_a^{n_k}$. Since the vector-valued activation function $\mt{\phi}$ is constructed by component-wise repetition of $\sigma$ and the permutation matrix $\mt{P}$ only rearranges the order of the state from $ [x_1 \ y_1 \ x_2 \ y_2 \ \cdots \ x_{n_k} \ y_{n_k}]^\top$ to $[x_1 \ \cdots \ x_{n_k} \ y_1 \ \cdots y_{n_k}]^\top$, we can easily check that $\mt{P}\mm{H}_a^{n_k} \subseteq \mm{G}(\mt{\phi}, \mm{I})$.  Hence $\mm{G}(\mt{\phi}, \mm{I}) = \mt{P}\mm{H}_a^{n_k} \subset \mt{P} \overline{\mm{H}}_a^{n_k}$ which implies $\overline{\mm{G}}(\mt{\phi}, \mm{X}_a) \supset {\mm{G}}(\mt{\phi}, \mm{X}_a)$. 
    % Since the vector-valued activation function $\mt{\phi}$ is constructed by component-wise repetition of $\sigma$, for any $(\mt{x}, \mt{y}) \in \mm{G}(\mt{\phi}, \mm{I})$, we have $ y_i = \sigma(\x_i), x_i \in \li -\alpha, \beta \ri, i = 1,2,\cdots,n_k$ where $x_i$ and $y_i$ are the $i$-th component of $\mt{x}$ and $\mt{y}$ respectively. Therefore $[x_1 \ y_1 \ x_2 \ y_2 \ \cdots \ x_{n_k} \ y_{n_k}]^\top \in \mm{H}_a^{n_k}$. By rearranging the order of all the $x_i$ and $y_i$, we have $[\mt{x}^\top \ \mt{y}^\top]^\top = [x_1 \ \cdots \ x_{n_k} \ y_1 \ \cdots y_{n_k}]^\top = \mt{P}[x_1 \ y_1 \ x_2 \ y_2 \ \cdots \ x_{n_k} \ y_{n_k}]^\top$. Hence $\mm{G}(\mt{\phi}, \mm{I}) \subseteq 
    % \mt{P}\mm{H}_a^{n_k}$. On the other hand, for any $(\mt{x},\mt{y}) \in \mt{P}\mm{H}_a^{n_a}$, we have $\mt{y} = \mt{\phi}(\mt{x})$ by vectorizing all the component-wise $\sigma$, which leads to $\mt{P}\mm{H}_a^{n_k} \subseteq \mm{G}(\mt{\phi}, \mm{I})$  Hence $\mm{G}(\mt{\phi}, \mm{I}) = \mt{P}\mm{H}_a^{n_k} \subset \mt{P} \overline{\mm{H}}_a^{n_k}$ which implies $\overline{\mm{G}}(\mt{\phi}, \mm{X}_a) \supset {\mm{G}}(\mt{\phi}, \mm{X}_a)$.
    Finally, since HZs are closed under the linear map and generalized intersection \cite[Proposition 7]{bird2023hybrid}, and both $\overline{\mm{H}}_a^{n_k}$ and $\mm{X}_a$ are HZs, we concludes that $\overline{\mm{G}}(\mt{\phi}, \mm{X}_a)$ is also an HZ, which completes the proof.
\end{proof}

Lemma \ref{lem:h_pi_over} provides the transformation from the scalar-valued nonlinear activation function $\sigma$ that is over-approximated based on Section \ref{subsec:nonlin} to vector-valued activation function $\mt{\phi}$, which is more favorable for layer-by-layer operations when constructing the graph set of FNNs. 

Denote the output set of Algorithm 1 in \cite{zhang2023backward} by replacing $\mm{H}$ with $\overline{\mm{H}}_a$ as $\mm{G}_{\mm{H}_\pi}$. The following theorem shows that $\mm{G}_{\mm{H}_\pi}$ is an over-approximation of the graph set of FNN $\mt{\pi}$ over the HZ-represented domain set $\mm{X}$.%Then, given the domain set represented as an HZ $\mm{X}$ and FNN $\mt{\pi}$, the following theorem shows that $\mm{G}_{\mm{H}_\pi}$ is an over-approximation of the graph set of FNN $\mt{\pi}$.

% \hang{The key of the proof is to show the over-approximation of the activation function is preserved under generalized intersection, cartesian product, and affine mapping.}
% Given the HZ domain $\mm{X}$, we apply affine transformation and Lemma \ref{lem:h_pi_over} layer-by-layer and then the over-approximation of $\mm{H}_\pi$ is returned by running Algorithm 1 in \cite{zhang2023backward} by replacing $\mm{H}$ with $\overline{\mm{H}}_a$, 
% which is shown in Theorem \ref{thm:hpi_over}.

\begin{theorem} \label{thm:hpi_over}
    Given an $\ell$-layer FNN $\mt{\pi}:\mathbb{R}^n \rightarrow \mathbb{R}^m$ with non-piecewise linear activation functions and the domain set represented as an HZ $\mathcal{X}\subset\mathbb{R}^n$, the output set $\mm{G}_{\mm{H}_\pi}$ is an over-approximation of the graph set of FNN $\mt{\pi}$, i.e., $\mm{G}_{\mm{H}_\pi}(\mm{X}, \mt{\pi}) = \overline{\mm{H}}_\pi \supset \mm{H}_\pi$.
\end{theorem}

\begin{proof}
    The proof boils down to showing that the output set $\mm{X}^{(k)}$ computed by Algorithm 1 in \cite{zhang2023backward} bounds all the possible outputs of the $k$-th layer of the FNN $\pi$, given the input set from last layer as $\mm{X}^{(k-1)}$, $k=1,2,\cdots,\ell - 1$.
    %It is easy to conclude that the proof of Theorem \ref{thm:hpi_over} is equivalent to prove that for each $k$-th layer, $k=1,2,\cdots,\ell - 1$, the output set $\mm{X}^{(k)}$ is an over-approximation of all the possible output states $\mt{x}^{(k)}$ given the input set from last layer $\mm{X}^{(k-1)}$.

    %Following the same procedures as Algorithm 1 in \cite{zhang2023backward}, there are three set operations to compute the $k$-th layer output set:
    By employing the identical steps outlined in Algorithm 1 of \cite{zhang2023backward}, three set operations are used to calculate the output set for the $k$-th layer:
     \begin{subequations}\label{equ:klayer}
         \begin{align}
          \mathcal{Z}^{(k-1)}&\!\leftarrow\!\bm{W}^{(k-1)}\mathcal{X}^{(k-1)}\!+\!\bm{v}^{(k-1)},  \label{equ:klayer_1}\\
        \mathcal{G}^{(k)}&\leftarrow(\bm{P}\cdot\overline{\mm{H}}_a^{n_{k}})\cap_{[\bm{I}_{n_k}\;\bm{0}]} \mathcal{Z}^{(k-1)},   \label{equ:klayer_2}\\
        \mathcal{X}^{(k)}&\leftarrow[\bm{0}\;\bm{I}_{n_k}] \cdot \mathcal{G}^{(k)}.      \label{equ:klayer_3}
         \end{align}
     \end{subequations}
     For any $\mt{x}^{(k-1)}  \in  \mm{X}^{(k-1)}$, let $\mt{z}^{(k-1)}$ be $\bm{W}^{(k-1)}\mt{x}^{(k-1)}\!+\!\bm{v}^{(k-1)}$. From \eqref{equ:klayer_1}, we have $\mt{z}^{(k-1)} \in \mm{Z}^{(k-1)}$. In terms of \eqref{equ:klayer_2} and Lemma \ref{lem:h_pi_over}, $[{\mt{z}^{(k-1)}}^\top \ {\mt{\phi}(\mt{z}^{(k-1)})}^\top]^\top \in \mm{G}^{(k)}$. From \eqref{equ:klayer_3}, $\mm{X}^{(k)}$ is the projection of $\mm{G}^{(k)}$ onto $\mt{\phi}(\mt{z}^{(k-1)})$ space, which implies that $\mt{x}^{(k)} \in \mm{X}^{(k)}$ by \eqref{equ:NN}. 
\end{proof}

Theorem \ref{thm:hpi_over} extends our previous work \cite{zhang2023backward} from the exact graph set of ReLU-activated FNN to the over-approximation of the graph set of FNN with more general activation functions. Let  $\overline{\mm{H}}_\pi = \overhz{{\pi}}$ be the over-approximated graph set of FNN $\mt{\pi}$ over the domain $\mathcal{X}$ using Theorem \ref{thm:hpi_over}, i.e., $\overline{\mm{H}}_\pi \supset \mm{H}_\pi$. With the same notations in Lemma \ref{coro:refine}, the following theorem provides the closed-form expression of an over-approximation of the one-step BRS for NFS \eqref{close-sys} with non-piecewise activation functions.
\begin{theorem} \label{thm:BRS_over_pi}
   Given the state set $\mm{X}$, the input set $\mm{U}$, the target set $\mm{T}$, the prior sets $\mm{X}^p \times \mm{U}^p$, and the over-approximated graph set of $\bm f$ given as  $\overline{\mathcal{H}}_f(\mathcal{X}^p,\mathcal{U}^p) = \hz{fp}$, 
  the one-step BRS of the NFS \eqref{close-sys} can be over-approximated by an HZ whose \emph{HCG}-representation is the same as \eqref{equ:thm_brs} by replacing $\hz{f}$ and $\hz{\pi}$ with $\hz{fp}$ and $\overhz{\pi}$, respectively.

       % Consider the state set $\mm{X}$, the input set $\mm{U}$ and a target set $\mm{T}$, the prior sets $\mm{X}^p \times \mm{U}^p$, and the over-approximated graph set of nonlinear dynamics $\overline{\mathcal{H}}_f(\mathcal{X}^p,\mathcal{U}^p) = \hz{fp}$ which are defined the identically to those in Corollary \ref{coro:refine}. Let  $\overline{\mm{H}}_\pi = \overhz{{\pi}}$ be the over-approximated graph set of FNN $\mt{\pi}$ over the domain $\mathcal{X}$ using Theorem \ref{thm:hpi_over}, i.e. $\overline{\mm{H}}_\pi \supseteq \mm{H}_\pi$. Then, an over-approximated one-step BRS of the neural feedback system \eqref{close-sys} can be computed as an HZ whose \emph{HCG}-representation is the same as \eqref{equ:thm_brs} by replacing $\hz{f}$ and $\hz{\pi}$ with $\hz{fp}$ and $\overhz{\pi}$ respectively.
\end{theorem}
\begin{proof}
    From the proof of Lemma \ref{coro:refine} we have 
    \begin{align*}
        \mathcal{P}(\mathcal{T}) &= \{\x \;|\; [\mt{x}^\top \ \u^\top\ \mt{y}^\top]^\top\in \mathcal{G}_f(\mathcal{X}^p,\mm{U}^p), \\
        & \ \quad \qquad [\mt{x}^\top \ \mt{u}^\top]^\top\in \mm{H}_\pi, \y\in \mathcal{T}\} \\
        &\subset \{\x \;|\; [\mt{x}^\top \ \u^\top\ \mt{y}^\top]^\top\in \mathcal{G}_f(\mathcal{X}^p,\mm{U}^p), \\
        & \ \quad \qquad [\mt{x}^\top \ \mt{u}^\top]^\top\in \overline{\mm{H}}_\pi, \y\in \mathcal{T}\}. 
    \end{align*}
    The remaining proof is the same as that of Theorem \ref{thm:brs}.
\end{proof}

%Utilizing Theorem \ref{thm:BRS_over_pi} for the over-approximation of one-step BRS, Algorithm \ref{alg:brs_comp} can be applied directly to compute the multi-step over-approximated BRSs for NFS \eqref{close-sys}.% with non-piecewise activation functions.

By using Theorem \ref{thm:BRS_over_pi} to over-approximate the one-step BRS in line \ref{alg:pre} of Algorithm \ref{alg:brs_comp}, the multi-step over-approximated BRSs for NFS \eqref{close-sys} with non-piecewise linear activation functions can be computed.

\section{Simulation Results} \label{sec:ex}
% In this section, we are studying a discrete-time neural feedback Duffing Oscillator system with different activation functions to illustrate the performance of the proposed method. 
In this section, we use two simulation examples to illustrate the performance of the proposed method. The considered  NFS consists of a Duffing Oscillator and an FNN controller with either ReLU (Example \ref{eg:relu_duffing}) or tanh (Example \ref{eg:tanh_duffing}) activation functions. 
The simulation examples are implemented in MATLAB R2022a and executed on a desktop with an Intel Core i9-12900k CPU and 32GB of RAM. 
%All the MILPs in this work are solved via Gurobi, which is a commercial solver demonstrating encouraging performance in tackling MILPs.
\begin{figure}[!t] 
\centering
 \begin{subfigure}{0.45\textwidth}
     \centering
     \includegraphics[width=0.97\textwidth]{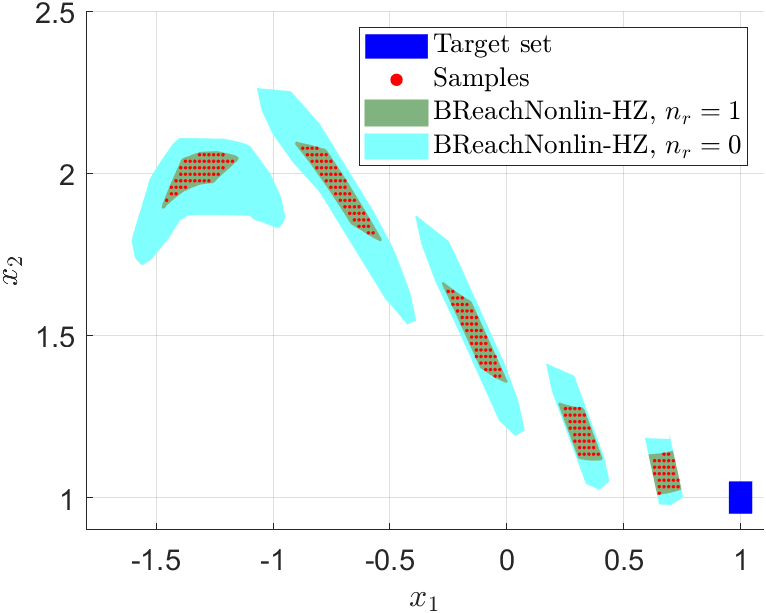}
     \caption{Over-approximated BRSs based on SOS.}
     % \caption{SOS approximation}
     % \label{fig:y equals x}
 \end{subfigure}
 \vskip 1mm
 \begin{subfigure}{0.45\textwidth}
     \centering
     \includegraphics[width=0.97\textwidth]{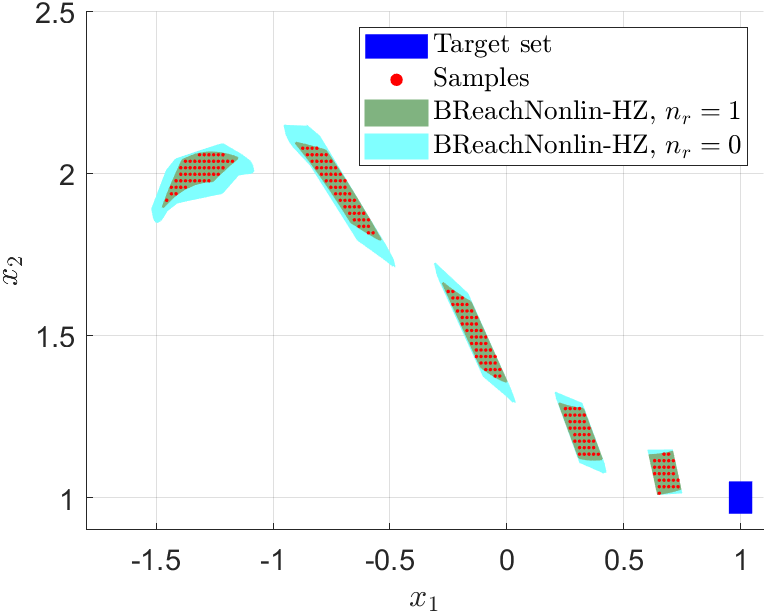}
     \caption{Over-approximated BRSs based on OVERT.}
     % \caption{OVERT approximation}
     % \label{fig:three sin x}
 \end{subfigure}
\caption{Over-approximated BRSs without refinement (cyan), the refined over-approximated BRSs (dark green), and samples (red) from the true BRSs for Example \ref{eg:relu_duffing}. }
%The top subfigure is based on SOS method while the bottom subfigure is based on the OVERT method.}
%The nonlinear dynamics are over-approximated by: (a) SOS approximation; (b) OVERT approximation.}
\label{fig:BRS_duff}
 % The over-approximation by OVERT (Left) is less conservative than that by SOS (right).
\end{figure}

\begin{figure*}[!t] 
\centering
\begin{subfigure}[t]{0.3\textwidth}
  \includegraphics[width=\linewidth]{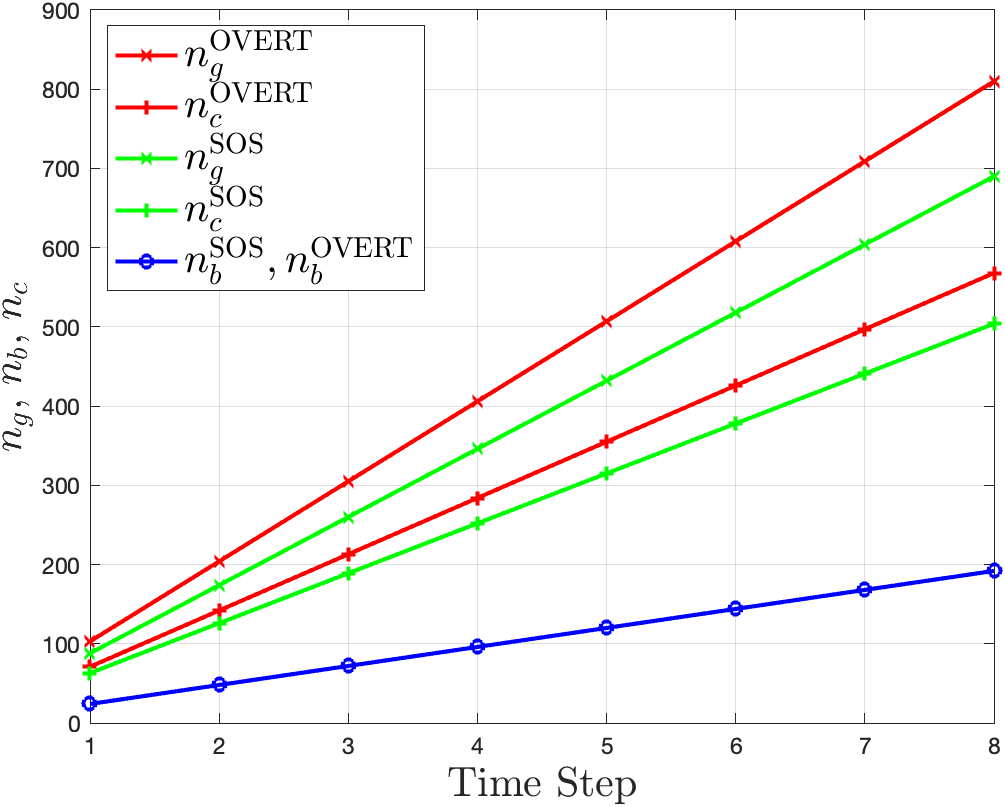}
  \caption{Set complexities.}
%    \caption{Computation time at each refinement epoch $n_r$ for Example \ref{eg:relu_duffing}.}
   \label{fig:set_complexity}
\end{subfigure}  \qquad
\begin{subfigure}[t]{0.3\textwidth}
    % \includegraphics[width=\linewidth]{figures/BRS_2steps_tanh.png}
    % \caption{Example \ref{eg:tanh_duffing}, BRS over-approximations.}
      \includegraphics[width=\linewidth]{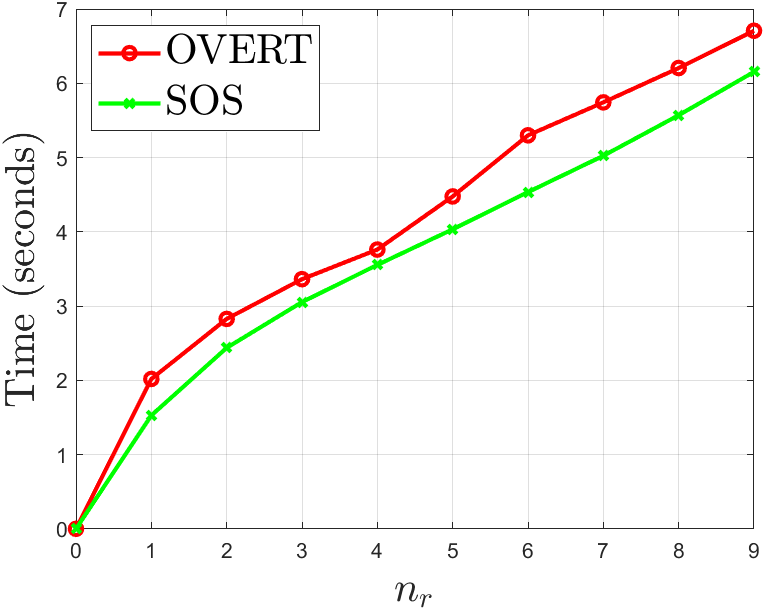}
      \caption{Computation time.}
  %\caption{Memory complexities of over-approximated BRSs for Example \ref{eg:relu_duffing}.}
  \label{fig:time_nr}
\end{subfigure} \qquad
\begin{subfigure}[t]{0.3\textwidth}
    \includegraphics[width=\linewidth]{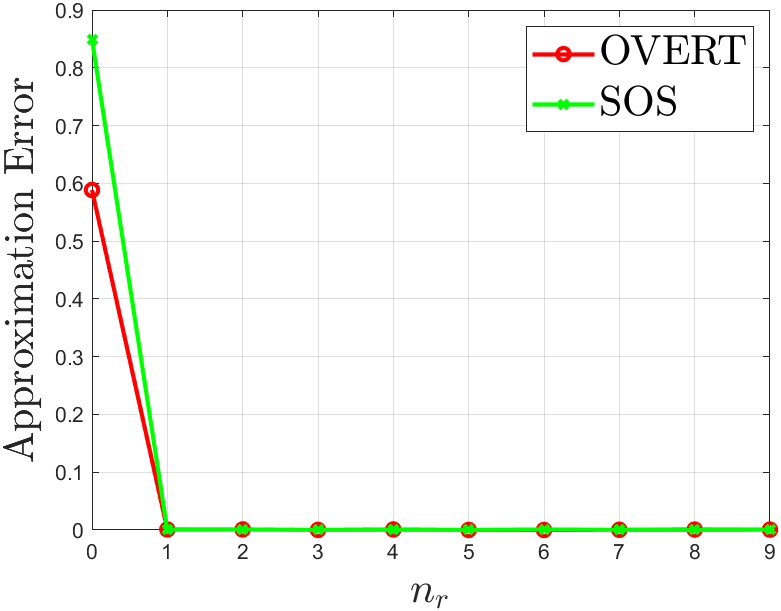}
    \caption{Approximation error.}
   % \caption{Approximation error at each refinement epoch $n_r$ for Example \ref{eg:relu_duffing}.}
    \label{fig:err_nr}
\end{subfigure}
\caption{Comparison of the SOS and OVERT approximation methods at each refinement epoch $n_r$ for Example \ref{eg:relu_duffing}.} %(left) computation time (left), the approximation error (center), and the memory complexities (right) }
%\XX{What's the y axis of (c)? Please make the three figures have the same height.}} \hang{y axis of (c) refers to the number of continuous generators, binary generators, and constraints, which does not have a unit.}
\label{fig:duffing_stat}
\vspace{-6mm}
\end{figure*}

\begin{example} \label{eg:relu_duffing}
Consider the following discrete-time Duffing Oscillator model from \cite{dang2012reachability}:
\begin{align*}
    x_1(t+1) & = x_1(t) + 0.3x_2(t), \\
    x_2(t+1) &= 0.3x_1(t) + 0.82x_2(t) - 0.3x_1(t)^3 + 0.3u(t),
\end{align*}
where $x_1$, $x_2 \in \mathbb{R}$ are the states confined in the state set $\mm{X} = \li -2,1.1 \ri \times \li -2,3 \ri$ and $u \in \mm{U} = \li 0,5\ri$ is the control input. The controller $u$ is a two-layer ReLU-activated FNN with 10-5 hidden neurons that is trained to learn the control policy given in \cite{dang2012reachability}. The exact BRSs are obtained by selecting samples from which trajectories can reach the target set within $T$ time steps; these samples are chosen from uniformly distributed grids over $\mm{X}$. To the best of our knowledge, there is no existing set-propagation-based method that can compute the over-approximated BRS for the system shown above.
    
Given the target set $\mm{T} = \li 0.95,1.05\ri\times \li 0.95,1.05\ri$, Algorithm \ref{alg:brs_comp} and Theorem \ref{thm:brs} can be employed to over-approximate the BRSs since the $\mm{X}$, $\mm{U}$ and $\mm{T}$ are HZs, satisfying all the conditions of Theorem \ref{thm:brs}. Note that we use 10 breakpoints to over-approximate the nonlinear term $\mt{x}_1^3$ via SOS and OVERT and then $\overline{\mm{H}}_f(\mm{X}_t^p, \mm{U}_t^p)$ is constructed by following a similar procedure as the example in \cite{siefert2023successor}. Fig. \ref{fig:BRS_duff} shows the over-approximated and exact BRSs, indicating that the BRSs computed by Algorithm \ref{alg:brs_comp} over-approximates the exact BRSs, as shown in Theorem \ref{thm:brs}, and the conservativeness is greatly mitigated by the refinement procedure.  In addition, the computed BRSs by using OVERT are less conservative than those by using SOS.

% Visualization results of the target set $\mm{T}$ (red), the over-approximation of BRSs $\overline{\mm{P}}_{1:2}(\mm{T})$, the refined over-approximation $\Tilde{\mm{P}}_{1:2}(\mm{T})$ by ReBReachNonlin-HZ (yellow), and the sample grids lying in the exact BRSs $\mm{P}_{1:2}(\mm{T})$ (red dots).

% Similar to, given $\overline{\mm{H}}_{f_n}$, the $\overline{\mm{H}}_f(\mm{X}^p)$ can be constructed by 
% \begin{align}
%     \mm{H}_{aug} &= \mat{1 & 0 \\ 0 & 1\\ 1 & 0.3 \\ 0 & 0}\mm{X}_0 \oplus \mat{0 \\ 0\\ 0\\ 1}\mm{H}_{b}(\mm{X}_0), \\
%     \mm{H}_{aug}' &= \mm{H}_{aug}' \cap_{\mat{1 & 0 & 0 & 0\\ 0 & 0 & 0 & 1}}\overline{\mm{H}}_{f_n}(\mm{X}_0) \\
%     \overline{\mm{H}}_f(\mm{X}_0) &= \mat{1 & 0 & 0 & 0\\ 0 & 1 & 0 & 0 \\ 0 & 0 & 1& 0\\ 0.3 & 0.82 & 0 & -0.3}\mm{H}_{aug}'
% \end{align}
% where $\mm{H}_b(\mm{X}_0)$ is a one-dimensional interval bounding $x_1^3$ over $\mm{X}_0$.

Fig. \ref{fig:set_complexity} shows the set complexity of the $t$-step BRS via SOS and OVERT for $t=1,2,\cdots,8$. Note that the BRSs and refined BRSs have the same set complexity. The number of generators and equality constraints has a linear growth rate with respect to time steps. In addition, the number of binary generators of BRSs is identical due to the same number of breakpoints used in the OVERT method and the SOS method. 

Fig. \ref{fig:time_nr} and Fig. \ref{fig:err_nr} show the computational time and the approximation error for five-step BRS with different numbers of refinement epochs $n_r$.
The approximation error is computed by the ratio between the volumes of the over-approximated BRS and exact BRS as $\text{error} = \frac{V_{\rm over} - V_{\rm exact}}{V_{\rm over}}$, where $V_{\rm over}$ and $V_{\rm exact}$ are the volume of over-approximation of last-step BRS and exact BRS respectively, which are approximated by the Monte Carlo method.
It can be observed that the computation time increases almost linearly when $n_r \geq 1$ because the refinement procedure does not affect the set complexity of the computed BRSs, and in addition, the approximation error is converged after just one epoch in this example.

\end{example}

\begin{example} \label{eg:tanh_duffing}
\begin{figure}[!ht]
\centering
\includegraphics[width=0.97\linewidth]{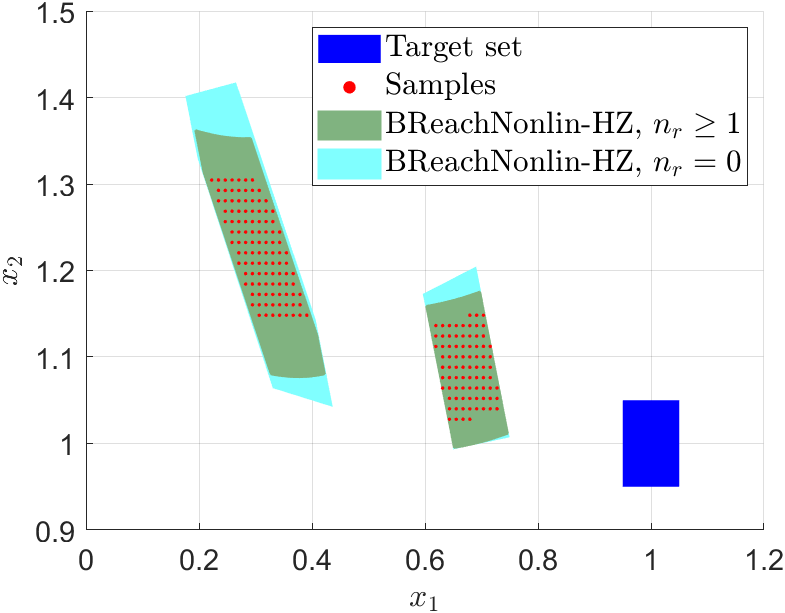}
\caption{Over-approximation of BRSs for Example \ref{eg:tanh_duffing}. The approximation error cannot be further reduced when $n_r > 1$.}
\label{fig:duf_tanh}
\end{figure}
Consider the same Duffing Oscillator model in Example \ref{eg:relu_duffing}. Instead of the ReLU-activated FNN controller, we train a tanh-activated FNN with 5-5 hidden neurons to learn the same control policy given in \cite{dang2012reachability}. We use OVERT with 7 breakpoints to construct the over-approximation of the graph set of FNN $\overline{\mm{H}}_\pi$ based on Theorem \ref{thm:hpi_over}, and then compute the over-approximated BRSs with the same configurations as those in Example \ref{eg:relu_duffing} based on Theorem \ref{thm:BRS_over_pi}. The simulation results are shown in Figure \ref{fig:duf_tanh}. It is observed that all the samples are enclosed by the over-approximated BRSs, consistent with Theorem \ref{thm:BRS_over_pi}.
\end{example}

\section{Conclusion}\label{sec:concl}
In this paper, we proposed a novel HZ-based approach to over-approximate BRSs of NFSs with nonlinear plant models. For NFSs with ReLU activation functions, by combining techniques that over-approximate nonlinear functions with the exact graph set of ReLU-activated FNNs, we showed the closed-form of over-approximated BRSs and a refinement procedure to reduce the conservativeness. In addition, we extended the results to NFSs with piecewise and non-piecewise linear activation functions. The performance of the proposed approach was evaluated using two numerical examples.
%\rev{Future work includes the order reduction techniques for HZs which be utilized to provide a better balance of computation complexity and efficiency for NFSs with complex NNs.}

\bibliographystyle{IEEEtran}
% \printbibliography
\bibliography{ref}

\end{document}